\documentclass[11pt,leqno]{amsart}
\usepackage[utf8]{inputenc}

\usepackage{amsmath}
\usepackage{amsthm}
\usepackage{amssymb}
\usepackage{enumerate}
\usepackage{mathtools}
 
\usepackage[colorlinks=false]{hyperref}

\usepackage{color}

\usepackage{xy}
\xyoption{all}

\numberwithin{equation}{section}

\theoremstyle{plain}
    \newtheorem{theorem}[equation]{Theorem}
    \newtheorem{lemma}[equation]{Lemma}
    \newtheorem{corollary}[equation]{Corollary}
    \newtheorem{proposition}[equation]{Proposition}

    \newtheorem*{theorem*}{Theorem}
    \newtheorem*{proposition*}{Proposition}
    \newtheorem*{corollary*}{Corollary}
    \newtheorem*{lemma*}{Lemma}
    \newtheorem*{conjecture*}{Conjecture}
    \newtheorem{definition-theorem}[equation]{Definition/Theorem}
    \newtheorem{definition-lemma}[equation]{Definition/Lemma}
\theoremstyle{definition}
    \newtheorem{definition}[equation]{Definition}
    \newtheorem{example}[equation]{Example}
    \newtheorem{examples}[equation]{Examples}

    \newtheorem{remark}[equation]{Remark}
    
    \newtheorem{remarks}[equation]{Remarks}

    \newcommand{\R}{\mathbb{R}}
    \newcommand{\C}{\mathbb{C}}

   	\renewcommand{\phi}{\varphi}
	\let\epsilon\varepsilon

    \newcommand{\Bounded}{\operatorname{B}}
    \newcommand{\CB}{\operatorname{CB}}
    \newcommand{\Compact}{\operatorname{K}}
    
    \newcommand{\cb}{\mathrm{cb}}
    
    \newcommand{\category}{\mathsf}
   \newcommand{\functor}{\mathcal}

\newcommand{\llangle}{\langle\!\langle}
\newcommand{\rrangle}{\rangle\!\rangle}
\newcommand{\into}{\hookrightarrow}

\newcommand{\restrict}{\big{\vert}}

\newcommand{\argument}{\hspace{2pt}\underbar{\phantom{g}}\hspace{2pt}}
\newcommand{\id}{\mathrm{id}}

\newcommand{\act}{\operatorname{act}}

\newcommand{\dual}{\vee}
\newcommand{\h}{\mathrm{h}}
\newcommand{\sigmah}{\sigma\mathrm{h}}
\newcommand{\eh}{\mathrm{eh}}
\renewcommand{\prod}{\bigsqcap}
\newcommand{\csmax}{}

    \DeclareMathOperator{\Res}{Res}
    
    \DeclareMathOperator{\Ind}{Ind}

	\DeclareMathOperator{\image}{image}
    
    \DeclareMathOperator{\GL}{GL}

	\DeclareMathOperator{\opsp}{OS}
	\DeclareMathOperator{\opmod}{OM}
	\DeclareMathOperator{\opcomod}{OC}

	\DeclareMathOperator{\cmod}{CM}
	\DeclareMathOperator{\ccomod}{CC}
	\DeclareMathOperator{\conn}{Con^f}
	\newcommand{\starRep}{\operatorname{Rep}}
	\newcommand{\adjoint}{\star}

\begin{document}

\title{Descent of Hilbert ${C^*}$-modules}
\date{November 1, 2017. Revised March 21, 2019.}
\author{Tyrone Crisp }
\thanks{Supported by a fellowship from the Radboud Excellence Initiative at Radboud University Nijmegen.}
\address{Institute for Mathematics, Astrophysics and Particle Physics, Radboud University Nijmegen, The Netherlands}
\curraddr{Department of Mathematics and Statistics, University of Maine, Neville Hall 333, Orono, ME 04469-5752, USA}
\email{tyrone.crisp@maine.edu}



\begin{abstract}
Let $F$ be a right Hilbert $C^*$-module over a $C^*$-algebra $B$, and suppose that $F$ is equipped with a left action, by compact operators, of a second $C^*$-algebra $A$. Tensor product with $F$  gives a functor from Hilbert $C^*$-modules over $A$ to Hilbert $C^*$-modules over $B$. We prove that under certain conditions (which are always satisfied if, for instance, $A$ is nuclear), the image of this functor can be described in terms of coactions of a certain coalgebra canonically associated to $F$. We then discuss several examples that fit into this framework: parabolic induction of tempered group representations; Hermitian connections on Hilbert $C^*$-modules; Fourier (co)algebras of compact groups; and the maximal $C^*$-dilation of operator modules over non-self-adjoint operator algebras.
\end{abstract}

\maketitle

\section{Introduction}

The notion of \emph{descent} originated in algebraic  geometry  as a technique for relating the geometry of a space $M$ to that of a covering space $N\to M$ (\cite{Grothendieck}; cf.~\cite[Chapter 6]{Neron}  for an exposition). The idea is to characterise those structures---vector bundles, or sheaves, for example---defined on $N$ that are pulled back from, and thus `descend' back down to, the base $M$. Translating geometry into algebra, one is led to consider questions such as that of characterising those $B$-modules that have the form $X\otimes_A B$ for some $A$-module $X$, given an inclusion of rings $A\into B$. In this paper we shall consider an analogue of this question for modules over algebras of Hilbert-space operators, and show that for a large class of algebras (including for instance all nuclear $C^*$-algebras), the descent problem has a particularly satisfactory solution: in the language of \cite{Grothendieck}, every injective $*$-homomorphism satisfies strict descent for Hilbert $C^*$-modules. 

The problem that we shall study is closely related to---indeed, as we shall presently explain, it is dual to---the one addressed by Rieffel in his $C^*$-algebraic generalisation of Mackey's imprimitivity theorem \cite[Theorem 6.29]{Rieffel_induced}. Let us recall that Rieffel's theorem describes the image of the functor
\begin{equation*}\label{eq:intro-Rieffel} 
\functor{F}_1:\starRep(B) \to \starRep(A) \qquad Y\mapsto F\otimes_B Y,
\end{equation*}
from the $*$-representations of a $C^*$-algebra $B$ to those of a second $C^*$-algebra $A$, given by tensor product with a $C^*$-correspondence  (or in the terminology of \cite{Rieffel_induced}, a Hermitian $B$-rigged $A$-module)  ${}_A F_B$. Rieffel observed that  the $A$-representations $F\otimes_B Y$ all carry a compatible representation of the $C^*$-algebra $K=\Compact_B(F)$ of $B$-compact operators on $F$, and he proved that the existence of such a representation of $K$ precisely characterises the image of $\functor{F}_1$. 

Alongside $\functor{F}_1$, each $C^*$-correspondence ${}_A F_B$ naturally determines a second functor 
\begin{equation*}\label{eq:intro-Kasparov} 
\functor{F}_2:\cmod(A) \to \cmod(B) \qquad X\mapsto X\otimes_A F,
\end{equation*}
from the category of right Hilbert $C^*$-modules over $A$ to those over $B$. Functors of this sort play a central role in operator $K$-theory (see \cite{Kasparov_icm} for a survey). An important instance of this is when the algebra $A$ acts on the correspondence $F$ by $B$-compact operators, in which case the functor $\functor{F}_2$ induces a map in $K$-theory $K_*(A)\to K_*(B)$. For example, if $N\to M$ is a surjection of compact Hausdorff spaces, then the corresponding inclusion of algebras of continuous functions $C(M)\into C(N)$ yields a $C^*$-correspondence ${}_{C(M)} C(N)_{C(N)}$ for which the functor $\functor{F}_2$ (and the induced map $K^*(M)\to K^*(N)$) is given by pull-back of continuous fields of Hilbert spaces from $M$ to $N$. The problem of describing the image of functors of the form $\functor{F}_2$ is thus an interesting one from the point of view of $K$-theory. 

This problem is also relevant to representation theory,   specifically to the tempered representation theory of reductive Lie groups. Given such a group $G$, and a  Levi subgroup $L\subseteq G$, Clare  showed in \cite{Clare} how to construct a $C^*$-correspondence ${}_{C^*_r(G)} F_{C^*_r(L)}$ whose associated tensor-product functor $\functor{F}_1$ is the well-known functor  of \emph{parabolic induction} of tempered unitary representations from $L$ to $G$. In \cite{CCH-Compositio} and \cite{CCH-JIMJ} we showed that the companion  functor $\functor{F}_2$ of \emph{parabolic restriction} also plays an important role  in representation theory. Detailed information about the image of this functor can be obtained, via the methods of \cite{CCH-Compositio} and \cite{CCH-JIMJ}, from deep representation-theoretic results due to Harish-Chandra, Langlands, and others. It would be of great interest to find an alternative, more geometric description of the image of parabolic restriction, since doing so might yield a new perspective on aspects of tempered representation theory. 

Motivated by these $K$-theoretic and representation-theoretic considerations, in this paper we shall prove (Theorem \ref{thm:Hermitian}) that the image of the functor $\functor{F}_2$ can, under certain conditions, be characterised in terms of \emph{co}actions  of an operator \emph{co}algebra associated to the correspondence $F$. Our approach relies on relating the functors $\functor{F}_1$ and $\functor{F}_2$ to one another in categorical terms, which was achieved in \cite{CH-Kadison} by embedding both categories $\starRep$ and $\cmod$ into the category $\opmod$ of \emph{operator modules} (cf.~\cite{BLM}). The functors $\functor{F}_1$ and $\functor{F}_2$ extend, using the \emph{Haagerup tensor product}, to functors between $\opmod(A)$ and $\opmod(B)$, and in \cite{CH-Kadison} we noted that if $A$ acts on $F$ by $B$-compact operators then these extended functors are  {adjoint} to one another. 

This adjunction allows us to bring a standard piece of categorical algebra---namely, the notion of  {(co)monads} and their categories of (co)modules---to bear on the problem of characterising the images of $\functor{F}_1$ and $\functor{F}_2$.  (See e.g., \cite[Chapter VI]{MacLane}, \cite[Chapter 3]{ttt}, or \cite[Chapter 4]{Borceux}, for the general theory.) Our description of the image of $\functor{F}_2$ is obtained through a study of the comonad  associated to  $\functor{F}_1$ and $\functor{F}_2$. Letting $F^\adjoint$ denote the operator $B$-$A$ bimodule adjoint to $F$, we find that the Haagerup tensor product $C=F^\adjoint \otimes^{\h}_A F$  carries a coalgebra structure; that the Hilbert $C^*$-$B$-modules $Y=X\otimes_A F$ all carry a compatible coaction $\delta_Y:Y\to Y\otimes^{\h}_B C$; and that under certain circumstances the existence of such a coaction  precisely characterises the image of $\functor{F}_2$. (Dually, Rieffel's imprimitivity theorem can be recovered through a study of the monad associated to $\functor{F}_1$ and $\functor{F}_2$.)

The appearance of  {co}algebras and  {co}modules in connection with $\functor{F}_2$, instead of the algebras and modules entering into Rieffel's imprimitivity theorem for $\functor{F}_1$, is easily explained  from a category-theoretic perspective: the functor $\functor{F}_2$   is \emph{left}-adjoint to  $\functor{F}_1$ (upon extension to operator modules), and left adjoints always produce comodules over comonads rather than modules over monads. The appearance of comodules can also be explained geometrically, by way of a simple example. Consider as above a continuous surjection $N\to M$ of compact Hausdorff spaces, and let $U$ be an open subset of $N$. Under what circumstances is $U$ the preimage of an open subset $V\subseteq M$, so that $C_0(U)\cong C_0(V)\otimes_{C(M)} C(N)$ as Hilbert $C^*$-modules over $C(N)$? An  obvious criterion is that the map $U\times_M N\to N$, ${(u,n)\mapsto n}$ should have image contained in $U$, and so give rise  (by pullback) to a map of $C(N)$-modules $C_0(U)\to C_0(U\times_M N)$. The latter map factors through a map $C_0(U)\to C_0(U)\otimes^{\h}_{C(M)} C(N)$,   which is a coaction of the kind appearing in our Theorem \ref{thm:Hermitian}.

Coalgebras and comodules related to the Haagerup tensor product have previously been studied from (at least) two other, rather different directions. In \cite{Effros-Ruan-Hopf}, Effros and Ruan showed that the multiplication on a von Neumann algebra $M$ dualises to a coproduct $M_\dual \to M_\dual\otimes^{\eh}_{\C} M_\dual$ on the predual $M_\dual$, where $\otimes^{\eh}$ is the \emph{extended} Haagerup tensor product. An important motivating example is that of $M=\operatorname{vN}(G)$, the von Neumann algebra of a locally compact group $G$, whose predual is the Fourier algebra $A(G)$ \cite{Eymard}. In certain  circumstances---e.g., if $M=\operatorname{vN}(G)$ for a  compact  group $G$ (or more generally, a compact quantum group: cf.~\cite[Chapter 9]{Daws})---Effros and Ruan's coproduct takes values in the ordinary Haagerup tensor product, and can thus be compared to the coalgebras studied in this paper. In Section \ref{subsec:Fourier} we show that in such  cases Effros and Ruan's coalgebras are isomorphic to coalgebras associated as above to $C^*$-correspondences (actually, to $*$-representations: $B=\C$). We conclude for instance that the tensor product with the regular representation  provides an equivalence between the category $\opmod(C^*(G))$ of operator modules over the $C^*$-algebra of a compact group $G$, and the category $\opcomod(A(G))$ of operator comodules over the Fourier (co)algebra of $G$. 

Comodules involving the Haagerup tensor product have also previously appeared in  operator $K$-theory, under the guise of \emph{connections} on operator modules. Following analogous constructions in algebra (cf.~\cite{Connes-NCG,Cuntz-Quillen-extensions}),  Mesland  in  \cite{Mesland} associated to each operator algebra $B$  the $B$-bimodule 
\[
\Omega(B)\coloneqq \ker\left (B\otimes^{\h}_{\C} B \xrightarrow{\textrm{multiplication}} B\right), 
\]
and he studied {connections} $\nabla: Z\to Z\otimes^{\h}_B \Omega(B)$ on operator $B$-modules. More generally, one can replace $\C$ by any closed nondegenerate subalgebra $A\subseteq B$ to obtain a bimodule $\Omega(B,A)$, and a corresponding notion of connection. In the parallel purely algebraic setting, given a connection $\nabla:Z\to Z\otimes_B \Omega(B,A)$ on a module over a ring $B$, one finds (cf. \cite{Nuss,Brz-grouplike}) that the map
\[
\delta: Z\xrightarrow{z\mapsto \nabla(z) + z\otimes 1\otimes 1} Z\otimes_B B\otimes_A B 
\]
defines a coaction of the (Sweedler) coring $C=B\otimes_A B$ if and only if the connection $\nabla$ is \emph{flat}; and  one obtains in this way an equivalence of categories between $C$-comodules and flat $\Omega(B,A)$-connections on $B$-modules. With a view to   applications in   $K$-theory, in Section \ref{subsec:connections} we extend these observations to the $C^*$-algebraic setting, and we deduce from Theorem \ref{thm:Hermitian} that for a nondegenerate inclusion of $C^*$-algebras $A\into B$, where $A$ is (for instance) nuclear, the image of the functor $\cmod(A)\to \cmod(B)$, $X\mapsto X\otimes_A B$---and hence, the image of the induced map in $K$-theory $K_*(A)\to K_*(B)$---can be described in terms of \emph{Hermitian connections} $Z\mapsto Z\otimes^{\h}_B \Omega(B,A)$ on Hilbert $C^*$-$B$-modules.

The organisation of the paper is as follows. In Section \ref{sec:preliminaries} we recall some background and establish our notation regarding operator algebras and their modules. In Section \ref{sec:coalgebras} we introduce the coalgebras and comodules that appear in our main results, which are proven in Sections \ref{sec:descent-om} (on operator modules) and \ref{sec:descent-Cmod} (on Hilbert $C^*$-modules). In Section \ref{sec:examples} we discuss several examples: parabolic induction of representations of reductive groups, following \cite{CCH-Compositio,CCH-JIMJ,CH-Kadison} (Section \ref{subsec:parabolic}); the coalgebras of Effros and Ruan (Section \ref{subsec:Fourier}); and Hermitian connections on Hilbert $C^*$-modules (Section \ref{subsec:connections}). Finally, in Section \ref{subsec:Cstarmax}, we contrast these $C^*$-algebraic examples  with a non-self-adjoint example, namely the inclusion  of a non-self-adjoint operator algebra into its maximal $C^*$-algebra (cf.~\cite{Blecher-cstarmax}).

Let us conclude this introduction with a few words on our proofs. We indicated above that the the formulation of our main result is obtained by `reversing the arrows' in (an appropriate  formulation of) Rieffel's imprimitivity theorem. The same is not true of the proofs, however, due to an inherent asymmetry of the Haagerup tensor product. A theorem of Beck  (cf.~\cite[Theorem 3.3.14]{ttt}) shows that a necessary condition for Rieffel's theorem to hold is that  the functor $\functor{F}_1$ (extended to operator modules) should preserve certain cokernels; while a necessary condition for our results to hold is that (the extension of) the functor $\functor{F}_2$ should preserve certain kernels. The Haagerup tensor product over a $C^*$-algebra in fact preserves all cokernels, but it does not preserve all kernels (see Section \ref{subsec:exactness}). Our verification that the functor $\functor{F}_2$ preserves the appropriate class of kernels (Lemma \ref{lem:WEP-heart}) adapts techniques from \cite{Mesablishvili-pure, JT, Borceux-Pelletier, Mesablishvili-Banach}---all based on some form of duality---to the context of operator modules. (Some extra care is required, owing to the fact that the operator-space dual of an operator module is not, in general, an operator module.) We also rely heavily on  a theorem of Anantharaman-Delaroche and Pop on the exactness of the Haagerup tensor product (\cite{An-Pop}, cf.~Theorem \ref{thm:exact}).

 \section{Preliminaries}\label{sec:preliminaries}

\subsection{Operator algebras and their modules}
 
We mostly adhere to the terminology and notation of \cite{BLM} with regard to operator algebras  and their modules. Let us briefly recall the definitions. 

An \emph{operator space} is a complex vector space $X$ equipped for each $n\geq 1$ with a Banach space norm on the matrix space $M_n(X)$, for which there exist a Hilbert space $H$ and an embedding $X\to \Bounded(H)$ such that the induced embeddings $M_n(X)\to M_n(\Bounded(H))\cong \Bounded(H^{\oplus n})$ are isometric for every $n$, where $\Bounded(H^{\oplus n})$ is given the operator norm.  

A  linear map of operator spaces $f:X\to Y$ is \emph{completely bounded}   if the quantity 
\[
\|f\|_{\cb}\coloneqq \sup_{n} \|M_n(f):M_n(X)\to M_n(Y)\| 
\]
is finite. The notions of \emph{completely contractive} and \emph{completely isometric} maps are defined in a similar way. A \emph{completely bounded isomorphism} is a completely bounded linear bijection whose inverse is also completely bounded. A \emph{complete embedding} is a map which has closed range and which is a completely bounded isomorphism onto its image. 
 
An \emph{operator algebra} is an operator space $A$ equipped with an associative bilinear product satisfying $\|aa'\|\leq \|a\|\cdot \|a'\|$ for every $n\geq 1$ and every $a,a'\in M_n(A)$.  We shall always assume that our operator algebras are  \emph{approximately unital}, meaning that they possess a contractive approximate unit. This latter assumption excludes many interesting examples of non-self-adjoint operator algebras; on the other hand, every $C^*$-algebra is an approximately unital operator algebra.

A (right) \emph{operator module} over an operator algebra $A$ is an operator space $X$ which is also a right $A$-module and whose norms satisfy $\|xa\|\leq \|x\|\cdot \|a\|$ for every $n\geq 1$ and every $a\in M_n(A)$ and $x\in M_n(X)$. In this case $A$ and $X$ can always be represented completely isometrically as spaces of Hilbert space operators, in such a way that the product in $A$ and the module action on $X$ become composition of operators. Important examples include Hilbert $C^*$-modules, which can be represented as Hilbert space operators via the `linking algebra' construction; and Hilbert-space representations $A\to \Bounded(H)$, for which the Hilbert space $H$ can be represented as rank-one operators on the conjugate Hilbert space $H^\adjoint$. Left operator modules, and   operator bimodules, are defined in a similar way to right modules.  {In this paper `module' means `right module' unless otherwise specified.}  An operator $A$-module $X$ is  \emph{nondegenerate} if every element of $X$ has the form $xa$ for some $a\in A$ and $x\in X$. Equivalently---by the Cohen factorisation theorem and our assumption that $A$ is approximately unital---$X$ is nondegenerate if the elements of the form $xa$ span a dense subspace of $X$. We shall work exclusively with nondegenerate modules.

If $E$ is an operator $A$-$B$ bimodule, where $A$ and $B$ are approximately unital operator algebras, and $X$ is a right operator $A$-module, then the \emph{Haagerup tensor product} $X\otimes^{\h}_A E$ is a right operator $B$-module. See \cite[1.5  \& 3.4]{BLM} for the definition and basic properties of the Haagerup tensor product.  If $X$ is a nondegenerate right (respectively, left) operator $A$-module, then the multiplication maps $X\otimes^{\h}_A A \to X$  (respectively, $A\otimes^{\h}_A X\to X$) are completely isometric isomorphisms, a fact that we shall frequently use without   comment. 

On operator modules over $C^*$-algebras there is a formal adjoint operation, which allows one to exchange left and right modules. If ${}_A X_B$ is an operator bimodule over $C^*$-algebras $A$ and $B$, then the adjoint $X^\adjoint$ is an operator $B$-$A$ bimodule: as a vector space $X^\adjoint$ is the complex-conjugate of $X$; the bimodule structure is defined by   $b\cdot x^\adjoint\cdot a \coloneqq (a^* xb^*)^\adjoint$ (where $x\mapsto x^\adjoint$ is the canonical conjugate-linear isomorphism $X\to X^\adjoint$); and the operator-space structure is given by the norms $\| [x^\adjoint_{ij}] \|_{M_n(X^\adjoint)}\coloneqq \| [x_{ji}] \|_{M_n(X)}$.  
The assignment $X\mapsto X^\adjoint$ is a functor, from operator $A$-$B$ bimodules to operator $B$-$A$ bimodules: each completely bounded bimodule map $t:X\to Y$ induces a completely bounded (with the same norm) bimodule map $t^\adjoint:X^\adjoint\to Y^\adjoint$ via $t^\adjoint(x^\adjoint)\coloneqq t(x)^\adjoint$. 
If $X$ is a right operator $A$-module, and $Y$ is a left operator $A$-module, then the map
\begin{equation}\label{eq:tensor-adjoint}
(X\otimes^{\h}_A Y)^\adjoint\to Y^\adjoint\otimes^{\h}_A X^\adjoint,\qquad (x\otimes y)^\adjoint\mapsto y^\adjoint\otimes x^\adjoint
\end{equation}
is a completely isometric isomorphism.

\subsection{The ${\opsp_1}$-category of operator modules}\label{subsec:opsp1}

For an approximately unital operator algebra $A$ we denote by $\opmod(A)$ the category whose objects are nondegenerate right operator $A$-modules, and whose morphisms are the completely bounded $A$-module maps. The space of such maps from $X$ to $Y$ will be denoted $\CB_A(X,Y)$.

It will be important in what follows to note that $\opmod(A)$ carries extra structure: its morphism sets $\CB_A(X,Y)$ are operator spaces, and composition of morphisms gives rise to completely contractive maps 
\[
\CB_A(Y,Z)\widehat{\otimes} \CB_A(X,Y) \to \CB_A(X,Z),
\]
where $\widehat{\otimes}$ is the projective tensor product of operator spaces (cf. \cite[1.2.19 \& 1.5.11]{BLM}). Let us refer to a category whose morphism sets and composition law satisfy the above conditions as an \emph{$\opsp_1$-category}. (This is an example of an \emph{enriched category}.) A functor $\functor{F}:\category{A}\to\category{B}$ of $\opsp_1$-categories will be called an \emph{$\opsp_1$-functor}, or a \emph{completely contractive functor}, if the induced maps on morphism sets are completely contractive linear maps of operator spaces. An example of an $\opsp_1$-functor is the functor 
\[
\opmod(A)\to \opmod(B),\qquad X\mapsto X\otimes^{\h}_A E
\]
of Haagerup tensor product with an operator bimodule; see \cite[Appendix]{Crisp-Frobenius} for a proof of this well-known fact.

Blecher has shown that the $\opsp_1$-category $\opmod(A)$ is a complete Morita invariant of $A$. To state this precisely, let us introduce the following terminology.

\begin{definition}\label{def:OS1-cats}
Let $\category{A}$ and $\category{B}$ be $\opsp_1$-categories. A \emph{complete contraction} in $\category{A}$ is a morphism $t$ with $\|t\|\leq 1$. We write $\category{A}_1$ for the subcategory of $\category{A}$ having the same objects as $\category{A}$, and having as morphisms the complete contractions in $\category{A}$. A \emph{completely isometric isomorphism} in $\category{A}$ is an isomorphism $t$ such that both $t$ and $t^{-1}$ are complete contractions: i.e., an isomorphism in $\category{A}_1$.
A \emph{completely isometric natural isomorphism} of functors $\functor{F},\functor{G}:\category{A}\to \category{B}$  is a natural isomorphism $\xi:\functor{F}\to \functor{G}$ such that for each $X\in \category{A}$ the morphism $\xi_X:\functor{F}(X)\to \functor{G}(X)$ is a completely isometric isomorphism. A \emph{completely isometric equivalence} between   $\category{A}$ and $\category{B}$ is an $\opsp_1$-functor $\functor{F}:\category{A}\to\category{B}$ for which there exists an $\opsp_1$-functor $\functor{G}:\category{B}\to\category{A}$ and completely isometric natural isomorphisms $\functor{F}\circ\functor{G}\cong \id_{\category{B}}$ and $\functor{G}\circ\functor{F}\cong \id_{\category{A}}$. 
\end{definition}

We note that if $\functor{F}:\category{A}\to \category{B}$ is an $\opsp_1$-functor, then $\functor{F}$ restricts to a functor $\functor{F}_1:\category{A}_1\to \category{B}_1$. If $\functor{F}$ is a completely isometric equivalence  then $\functor{F}_1$ is an equivalence.

The following result is due to Blecher \cite[Theorem 1.2]{Blecher-cstarMorita}.

\begin{theorem}\label{thm:Blecher-Morita}
If $A$ and $B$ are $C^*$-algebras and $\functor{F}:\opmod(A)\to \opmod(B)$ is a completely isometric equivalence, then $\functor{F}$ is completely isometrically isomorphic to the functor of Haagerup tensor product with a Morita equivalence bimodule.
\end{theorem}

In Section \ref{subsec:Cstarmax} we shall make use of the categorical notion of \emph{kernels} and \emph{cokernels}. Recall (e.g., from \cite[VIII.1]{MacLane}) that a \emph{kernel} of a morphism $t:X\to Y$ in a category $\category{A}$ with a zero object is a morphism $i:I\to X$ such that $t\circ i=0$, and such that any other morphism $j:J\to X$ having $t\circ j=0$ factors uniquely as a composition $J\to I\xrightarrow{i} X$. The notion of a \emph{cokernel} of $t$ is defined analogously as a map $q:Y\to Q$ which satisfies $q\circ t=0$ and which is universal for this property. Kernels and cokernels, when they exist, are unique up to isomorphism.

\begin{lemma}\label{lem:ker-coker}
Let $A$ be an operator algebra and let $t:X\to Y$ be a morphism in $\opmod(A)$ (respectively, in $\opmod(A)_1$). The inclusion $i:\ker(t)\to X$ is a kernel of $t$, and the quotient mapping $q:Y\to Y/\overline{\image(t)}$ is a cokernel of $t$ in $\opmod(A)$ (respectively, in $\opmod(A)_1$). In particular, the kernels in $\opmod(A)$ are precisely the complete embeddings, while the kernels in $\opmod(A)_1$ are precisely the complete isometries.
\end{lemma}

\begin{proof}
The maps $i$ and $q$ are easily seen to possess the necessary universal properties (cf.~\cite[1.2.15]{BLM}). The last assertion follows from the uniqueness of kernels: every kernel of $t$ in $\opmod(A)$ (respectively, $\opmod(A)_1$) is conjugate, by a completely bounded (respectively, completely isometric) isomorphism, to the complete isometry $i:\ker(t)\to X$. 
\end{proof}

\subsection{Exactness of the Haagerup tensor product}\label{subsec:exactness}

The simplicity of the descent theorem for operator modules over $C^*$-algebras, compared to other kinds of algebras, is largely due to the following exactness property of the Haagerup tensor product, established by Anantharaman-Delaroche and Pop \cite[Corollary, p.411]{An-Pop}

\begin{theorem}\label{thm:exact}
Let $A$ be a $C^*$-algebra, let $X$ be a right operator $A$-module, and let $Y$ be a left operator $A$-module. Let $I$ be a closed submodule of $X$, and let $i:I\to X$ and $q:X\to X/I$ denote the inclusion and the quotient mappings, respectively. Consider the maps
\[
I\otimes^{\h}_A Y \xrightarrow{i\otimes \id_Y} X\otimes^{\h}_A Y \xrightarrow{q\otimes \id_Y} (X/I) \otimes^{\h}_A Y.
\]
\begin{enumerate}[\rm(a)]
\item The map  $i\otimes \id_Y$ is a complete isometry.
\item The map  $q\otimes \id_Y$ is a complete quotient mapping; that is, the induced map $(X\otimes^{\h}_A Y)/\ker(q\otimes \id_Y)\to (X/I)\otimes^{\h}_A Y$ is a completely isometric isomorphism.
\item The image of $i\otimes \id_Y$ equals the kernel of $q\otimes \id_Y$.
\end{enumerate}
\end{theorem}

\begin{proof}
The assertions (a) and (b) are proved in \cite{An-Pop} (cf.~\cite[1.5.5 \& 3.6.5]{BLM}). Part (c) is easier, and presumably well known: we clearly have an inclusion $\image(i\otimes \id_Y)\subseteq \ker(q\otimes \id_Y)$, whence a mapping
\begin{equation}\label{eq:exactness}
(X\otimes^{\h}_A Y)/ \image(i\otimes \id_Y) \to (X/I)\otimes^{\h}_A Y. 
\end{equation}
It is a straightforward matter to check that the formula
\[
(X/I) \times Y \to (X\otimes^{\h}_A Y)/ \image(i\otimes \id_Y),\quad (x+I,y)\mapsto (x\otimes y)+\image(i\otimes \id_Y)
\]
defines a completely contractive (in the sense of \cite[1.5.4]{BLM}) $A$-balanced bilinear map, and hence induces a map on $(X/I)\otimes^{\h}_A Y$ that is inverse to \eqref{eq:exactness}.
\end{proof}

\begin{remark}
Note that the equality $\ker(t\otimes \id_Y)=\ker(t) \otimes^{\h}_A Y$ is not generally valid if $t$ is not a quotient map. For example, let $B$ be a $C^*$-algebra containing $A$ as a subalgebra, and let $t:A\to A$ be given by $a\mapsto a_0 a$, where $a_0$ is an element of $A$ which is not a zero-divisor in $A$, but which is a zero-divisor in $B$. Then the kernel of $t$ is trivial, whereas the kernel of $t\otimes \id_B : A\otimes^{\h}_A B \to A\otimes^{\h}_A B$ is not trivial. Examples of this kind clearly do not arise when $A=\C$, and indeed Blecher and Smith have shown that in this case one does have $\ker(t\otimes \id_Y)= \ker(t)\otimes^{\h} Y$ for every completely bounded map $t$: see \cite[Remark, p.137]{Blecher-Smith} and \cite[Theorem 2.4]{Allen-Sinclair-Smith}. We thank David Blecher for helpful discussions of this and other aspects of the exactness of $\otimes^{\h}$.
\end{remark}

\subsection{Weak expectations}

As a final terminological preliminary, we recall (e.g., from \cite[Section 3.6]{Brown-Ozawa}; cf.~\cite{Lance}) that a \emph{weak expectation} for an inclusion of $C^*$-algebras $A\into D$ is a contractive, completely positive map $\iota:D\to A^{\dual\dual}$ (where $\dual$ indicates the dual Banach space) extending the canonical embedding $A\into A^{\dual\dual}$. Let us remark that such an $\iota$ is automatically a completely contractive $A$-bimodule map (see \cite[Section 1.5]{Brown-Ozawa}). A $C^*$-algebra $A$ has Lance's \emph{weak expectation property} if   every inclusion $A\into D$ admits a weak expectation.   Every nuclear $C^*$-algebra has the weak expectation property, as do many non-nuclear algebras. It is an open problem---equivalent, as shown by Kirchberg \cite{Kirchberg}, to Connes's embedding problem---to determine whether the $C^*$-algebras of  nonabelian free groups have the weak expectation property. Note that non-self-adjoint operator algebras never have (the appropriate non-self-adjoint analogue of) this property: see \cite[Theorem 7.1.7]{BLM}.

\section{The coalgebra of an adjoint pair of bimodules}\label{sec:coalgebras}

This section introduces the basic operator-algebra structures---operator coalgebras and their comodules---that appear in our descent theorems for Hilbert $C^*$-modules and operator modules. While we are mostly interested in coalgebras associated to $C^*$-correspondences,  interesting examples do arise in other contexts (see, e.g., Section \ref{subsec:Cstarmax}), and in order to cover such examples we work throughout Sections \ref{sec:coalgebras} and \ref{sec:descent-om} in the more general setting described below.  

\subsection{Adjoint pairs of operator bimodules}\label{subsec:adjoint-pairs}
 
\begin{definition}\label{def:adjunction}
Let $A$ and $B$ be approximately unital operator algebras,  and let ${}_A L_B$ and ${}_B R_A$ be a pair of nondegenerate operator bimodules. We shall call $(L,R)$ an \emph{adjoint pair of bimodules} if we are given completely contractive  bimodule maps
\[
\eta:A\to L\otimes^{\h}_B R\qquad \text{and}\qquad \epsilon:R\otimes^{\h}_A L  \to B
\]
(referred to, respectively, as the \emph{unit} and \emph{counit}) for which the diagrams
\begin{equation}\label{eq:triangle}
\xymatrix@C=40pt@R=40pt{ 
R\otimes^{\h}_A A \ar[r]^-{\id_R\otimes\eta} \ar[d]_-{\cong} & R\otimes^{\h}_A L\otimes^{\h}_B R \ar[d]^-{\epsilon\otimes\id_R} \\
R & B\otimes^{\h}_B R \ar[l]_-{\cong} 
}
\qquad 
\xymatrix@C=40pt@R=40pt{
A\otimes^{\h}_A L \ar[r]^-{\eta\otimes\id_L} \ar[d]_-{\cong} & L\otimes^{\h}_B R\otimes^{\h}_A L \ar[d]^-{\id_L\otimes\epsilon} \\ L & L\otimes^{\h}_B B \ar[l]_-{\cong}
}
\end{equation}
commute. In this situation we will use the notation $K\coloneqq L\otimes^{\h}_A R$ and $C\coloneqq R\otimes^{\h}_A L$.
\end{definition}
 
To explain the choice of terminology, let us consider for each $X\in \opmod(A)$ the map $\eta_X:X\to X \otimes^{\h}_A K$ defined as the composition
\[
\eta_X:X\xrightarrow{\cong} X\otimes^{\h}_A A \xrightarrow{\id_X\otimes \eta} X \otimes^{\h}_A K,
\]
and for each $Y\in \opmod(B)$ the map $\epsilon_Y:Y\otimes^{\h}_B C\to Y$ defined as the composition
\[
\epsilon_Y:Y\otimes^{\h}_B C\xrightarrow{\id_Y\otimes \epsilon} Y\otimes^{\h}_B B
\xrightarrow{\cong} Y.
\]
The maps $\eta_X$ and $\epsilon_Y$ are natural in $X$ and $Y$ (respectively), meaning that for all morphisms $t\in \CB_A(X,X')$ and $s\in \CB_B(Y,Y')$ the diagrams
\begin{equation}\label{eq:natural}
\xymatrix@C=40pt{
X \ar[r]^-{t} \ar[d]_-{\eta_X} & X' \ar[d]^-{\eta_{X'}} \\
X\otimes^{\h}_A K \ar[r]^-{t\otimes \id_K} & X'\otimes^{\h}_A K
}
\qquad
\xymatrix@C=40pt{
Y \ar[r]^-{s} & Y' \\
Y\otimes^{\h}_B C \ar[r]^-{s\otimes \id_C} \ar[u]^-{\epsilon_Y} & Y'\otimes^{\h}_B C \ar[u]_-{\epsilon_{Y'}}
}
\end{equation}
commute. These natural maps are the unit and counit of a completely isometric adjunction between the Haagerup tensor product functors 
\[
\functor{L}:\opmod(A)\xrightarrow{X\mapsto X\otimes^{\h}_A L} \opmod(B) \quad \text{and}\quad \functor{R}:\opmod(B)\xrightarrow{Y\mapsto Y\otimes^{\h}_B R} \opmod(A).
\]
That is to say, the map
\[
\CB_B(X\otimes^{\h}_A L, Y) \xrightarrow{t\mapsto (t\otimes \id_R)\circ \eta_X} \CB_A(X, Y\otimes^{\h}_B R)
\]
is a completely isometric natural isomorphism, whose inverse is given by $s\mapsto \epsilon_Y\circ (s\otimes \id_{L})$. See \cite[Chapter IV]{MacLane} for the general notion of adjoint functors.

\begin{example}[$C^*$-correspondences]\label{ex:corresp}
Let $A$ and $B$ be $C^*$-algebras, and let ${}_A F_B$ be a \emph{$C^*$-correspondence} from $A$ to $B$: that is, a right Hilbert $C^*$-module over $B$ equipped with a $*$-homomorphism from $A$ into the $C^*$-algebra of adjointable operators on $F$. (See \cite{Lance} or \cite[Chapter 8]{BLM} for the relevant background on Hilbert $C^*$-modules.) Equip $F$ with its canonical operator space structure (cf.~\cite[8.2]{BLM}). Blecher has shown that the tensor product functor $\otimes^{\h}_A F:\opmod(A)\to \opmod(B)$ coincides, on the subcategory of Hilbert $C^*$-modules, with the functor $\functor{F}_2$ considered in the introduction; see \cite[Theorem 8.2.11]{BLM}.

Suppose now that the image of $A$ under the action homomorphism lies in the ideal $\Compact_B(F)$ of \emph{compact} operators on $F$. Let $F^\adjoint$ denote the operator $B$-$A$ bimodule adjoint to $F$. Another theorem of Blecher \cite[Corollary 8.2.15]{BLM} identifies $\Compact_B(F)$ with the Haagerup tensor product $K=F\otimes^{\h}_B F^\adjoint$, and in \cite{CH-Kadison} we observed that the unit 
\[ 
\eta: A\xrightarrow{\textrm{action}} \Compact_B(F) \cong F\otimes^{\h}_B F^\adjoint
\]
and the counit
\[
\epsilon: F^\adjoint\otimes^{\h}_B F \xrightarrow{x^\adjoint\otimes y\mapsto \langle x|y\rangle} B
\]
make $(F,F^\adjoint)$ into an adjoint pair of operator bimodules. See \cite{CH-Kadison} and \cite{Crisp-Frobenius} (where the opposite convention regarding left versus right modules is used) for more on this example. See also  \cite{Blecher-Kaad-Mesland} for an analogue of $C^*$-correspondences, and `compact' operators thereon, for non-selfadjoint operator $*$-algebras.
\end{example}

\begin{example}[Subalgebras]\label{ex:subalg}
Let $B$ be an operator algebra and let $A\subseteq B$ be a norm-closed subalgebra, which is nondegenerate in the sense that $AB=B=BA$. Then the unit 
\[
\eta: A\xrightarrow{\textrm{inclusion}} B \cong B\otimes^{\h}_B B
\]
and the counit
\[
\epsilon: B\otimes^{\h}_A B \xrightarrow{\textrm{multiplication}} B
\]
make $({}_A B_B, {}_B B_A)$ into an adjoint pair of operator bimodules. The corresponding functors are the restriction functor $\opmod(B)\to \opmod(A)$, and its left-adjoint induction functor $\opmod(A)\xrightarrow{X\mapsto X\otimes^{\h}_A B}\opmod(B)$. Note that if $A$ and $B$ are $C^*$-algebras then this example is an instance of the previous one: $B$ is a Hilbert $C^*$-module over itself, with inner product $\langle b_1 | b_2\rangle = b_1^* b$, and we have $\Compact_B(B)=B$ (acting by left multiplication). Note further that the operator $B$-bimodule $C=B\otimes^{\h}_A B$  can in this case be identified with the closed linear span of the products $b_1*b_2$ in the amalgamated free product $C^*$-algebra $B*_A B$; see \cite[Theorem 3.1]{CES}, \cite[Lemma 1.14]{Pisier}, \cite[p.515]{Ozawa}.
\end{example}

\begin{remark}
We insist upon $\eta$ and $\epsilon$ being completely contractive, rather than merely completely bounded, in order to ensure that we eventually obtain completely isometric equivalences of $\opsp_1$-categories (cf.~Theorem \ref{thm:Blecher-Morita}).   As is often the case in operator algebra theory (see \cite[Chapter 5]{BLM}, for example), much of what we do below admits a parallel `completely bounded' version. We shall not say any more about that here, except to point out an example: the unit of the adjunction in \cite[Theorem 4.15]{CH-Kadison} is not completely contractive. (The unit can be made completely contractive by a simple rescaling, but then the counit would cease to be completely contractive.)
\end{remark}

\subsection{Operator coalgebras and comodules}\label{subsec:coalgebras}

We temporarily suspend the standing notation from the previous section to make a general definition.

\begin{definition}\label{def:coalgebra}
Let $B$ be an approximately unital operator algebra. An \emph{operator $B$-coalgebra} is a nondegenerate operator $B$-bimodule $C$ equipped with completely contractive $B$-bimodule maps 
\[
\delta: C\to C\otimes^{\h}_B C\qquad \textrm{and}\qquad \epsilon:C\to B
\]
making the diagrams
\[
\xymatrix@C=40pt{ 
C \ar[r]^-{\delta} \ar[d]_-{\delta} & C\otimes^{\h}_B C \ar[d]^-{\delta\otimes \id_C} \\ 
C\otimes^{\h}_B C \ar[r]_-{\id_C\otimes\delta} & C\otimes^{\h}_B C\otimes^{\h}_B C 
}
\qquad 
\xymatrix@C=40pt{
C\otimes^{\h}_B C \ar[d]_-{\id_C\otimes\epsilon} & C \ar[l]_-{\delta} \ar[r]^-{\delta} \ar[d]^-{\id_C} & C\otimes^{\h}_B C \ar[d]^-{\epsilon\otimes\id_C} \\
C\otimes^{\h}_B B \ar[r]^-{\cong} & C & B\otimes^{\h}_B C \ar[l]_-{\cong}
}
\]
commute.

A  \emph{right operator $C$-comodule} is a pair $(Z,\delta_Z)$ consisting of a nondegenerate right  operator $B$-module $Z$ and a completely contractive $B$-module map $\delta_Z:Z\to  Z\otimes^{\h}_B C$ making the diagrams 
\begin{equation}\label{eq:comodule-axioms}
\xymatrix@C=40pt{ 
Z \ar[r]^-{\delta_Z} \ar[d]_-{\delta_Z} & Z\otimes^{\h}_B C \ar[d]^-{\id_Z\otimes \delta} \\ 
Z\otimes^{\h}_B C \ar[r]_-{\delta_Z\otimes \id_C} & Z\otimes^{\h}_B C\otimes^{\h}_B C 
}
\qquad\text{and}\qquad 
\xymatrix@C=40pt{
Z \ar[r]^-{\delta_Z} \ar[d]_-{\id_Z} & Z\otimes^{\h}_B C \ar[dl]^-{\epsilon_Z} \\
Z & 
}
\end{equation}
commute. The space of morphisms $\CB_C(Z,W)$ between operator $C$-comodules $(Z,\delta_Z)$ and $(W,\delta_W)$ is defined to be the set of maps $t\in \CB_B(Z,W)$ for which the diagram 
\begin{equation}\label{eq:comodule-morphism}
\xymatrix@C=40pt{
Z  \ar[d]_-{t} \ar[r]^-{\delta_{Z}} & Z\otimes^{\h}_B C  \ar[d]^-{t\otimes \id_C} \\
W \ar[r]^-{\delta_{W}} & W\otimes^{\h}_B C
}
\end{equation}
commutes. Let $\opcomod(C)$ denote the category of right operator $C$-comodules, with morphisms $\CB_C(Z,W)$ as above. We denote by $\opcomod^{\cb}(C)$ the category whose objects are pairs $(Z,\delta_Z)$, where $Z$ is an operator $B$-module, and $\delta_Z:Z\to Z\otimes^{\h}_B C$ is a completely \emph{bounded} $B$-module map making the diagrams \eqref{eq:comodule-axioms} commute. The morphisms in $\opcomod^{\cb}(C)$ are the completely bounded $B$-module maps making the diagram \eqref{eq:comodule-morphism} commute.
\end{definition}

\begin{remarks}\label{rem:coalgebras-comodules}\begin{enumerate}[\rm(a)]
 \item The morphism spaces $\CB_C(Z,W)$ in $\opcomod(C)$ and $\opcomod^{\cb}(C)$ are closed subspaces of the operator spaces $\CB_B(Z,W)$; hence $\opcomod(C)$ and $\opcomod^{\cb}(C)$ are an $\opsp_1$-categories. 
\item The coproduct $\delta$ for an operator coalgebra is a complete isometry: indeed, $\delta$ is a complete contraction, and is split by the complete contraction $\id_C\otimes \epsilon$. Similarly, the coaction $\delta_Z$ for an object in $\opcomod(C)$ (respectively, in $\opcomod^{\cb}(C)$) is a complete isometry (respectively, a complete embedding). 
\item The notions of left  comodules and  bi-comodules, and of morphisms between them, are defined via the obvious modifications of the above definition. Following our convention for modules, `comodule'  means `right comodule' in the absence of further specification.
\item  
The category $\opcomod^{\cb}(C)$ is the Eilenberg-Moore category for the comonad $\functor{C}\coloneqq \otimes^{\h}_B C$ on $\opmod(B)$.  (See \cite[Chapter VI]{MacLane} for the terminology.) Since the functor $\functor{C}$ and the structure maps $\delta$ and $\epsilon$ are completely contractive, $\functor{C}$ restricts to a comonad $\functor{C}_1$ on $\opmod(B)_1$, whose Eilenberg-Moore category is  $\opcomod(C)_1$. The category $\opcomod(C)$ thus sits between the Eilenberg-Moore categories for $\functor{C}_1$ and $\functor{C}$: it has the same objects as the former, and the same morphisms as the latter.
\end{enumerate}
\end{remarks}

\subsection{The coalgebra of an adjoint pair}\label{subsec:adjoint-coalgebra}

We now return to the setting of Section \ref{subsec:adjoint-pairs}: $A$ and $B$ are approximately unital operator algebras, and $({}_A L_B,{}_B R_A)$ is an adjoint pair of operator bimodules with unit $\eta:A\to K\coloneqq L\otimes^{\h}_B R$ and counit $\epsilon:C\coloneqq R\otimes^{\h}_A L\to B$.

\begin{definition}\label{def:C}
On the operator $B$-bimodule $C=R\otimes^{\h}_A L$ we consider the coproduct $\delta: C\to C\otimes^{\h}_B C$
defined by
\[
\delta\coloneqq \eta_R\otimes \id_L:R \otimes^{\h}_A L  \to R  \otimes^{\h}_A L \otimes^{\h}_B R  \otimes^{\h}_A L. 
\]
The associativity of the Haagerup tensor product ensures that $\delta$ is coassociative, while  the commutativity of the diagrams \eqref{eq:triangle} ensures that 
\[
(\epsilon\otimes \id_C)\circ \delta = (\id_C\otimes \epsilon)\circ \delta = \id_C.
\]
  Thus the counit $\epsilon$ and the coproduct $\delta$ furnish $C$ with the structure of  an operator $B$-coalgebra.
\end{definition}

\begin{examples}\label{ex:comodule}
 \begin{enumerate}[\rm(a)]
\item The map $\delta_{L}: L  \to L\otimes^{\h}_A C$ defined as the composition
\[
L\xrightarrow{\cong} A\otimes^{\h}_A L \xrightarrow{\eta\otimes\id_L} K\otimes^{\h}_A L = L\otimes^{\h}_B C
\] 
furnishes $L $ with the structure of a right operator $C$-comodule. Similarly the map $\delta_R\coloneqq \eta_R$ makes $R$ a left operator $C$-comodule.
\item If $X$ is any operator $A$-module, then setting 
\[
\delta_{X \otimes^{\h}_A L}\coloneqq \eta_X\otimes \id_L = \id_X\otimes \delta_{L}
\]
yields an operator $C$-comodule $(X\otimes^{\h}_A L, \delta_{X\otimes^{\h}_A L})$.
\end{enumerate}
\end{examples}
 
\begin{definition}\label{def:comodules-comparison}
We consider the \emph{comparison functor} $\functor{L}_C: \opmod(A)\to \opcomod(C)$ defined on objects by $X\mapsto (X\otimes^{\h}_A L,\delta_{X\otimes^{\h}_A L})$ (see Example \ref{ex:comodule}), and on morphisms by $t\mapsto t\otimes \id_L$.
\end{definition}
 
The comparison functor $\functor{L}_C$ is given on morphisms by a Haagerup tensor product, so it is an $\opsp_1$-functor. Note that there is a forgetful functor $\opcomod(C)\to \opmod(B)$, $(Z,\delta_Z)\mapsto Z$, making the diagram
\begin{equation}\label{eq:comparison-diagram}
\xymatrix@C=40pt{
& \opcomod(C) \ar[d]^-{\textrm{forget}} \\
\opmod(A) \ar[r]_-{\functor{L}} \ar[ur]^-{\functor{L}_C} & \opmod(B)
}
\end{equation}
commute.

 \section{Descent of operator modules}  \label{sec:descent-om}
  
 The notation in this section is the same as in the previous one: $A$ and $B$ are approximately unital operator algebras, and $({}_A L_B,{}_B R_A)$ is an adjoint pair of operator bimodules with unit $\eta:A\to K\coloneqq L\otimes^{\h}_B R$ and counit $\epsilon:C\coloneqq R\otimes^{\h}_A L\to B$. 

\begin{theorem}\label{thm:WEP}
Suppose that $A$ is a $C^*$-algebra, and that there is a completely contractive $A$-bimodule map $\iota:K=L\otimes^{\h}_B R\to A^{\dual\dual}$ such that $\iota\circ \eta:A\to A^{\dual\dual}$ is the canonical embedding.  Then the comparison functor $\functor{L}_C:\opmod(A)\to \opcomod(C)$ is a completely isometric equivalence.
\end{theorem}
 
\begin{remarks}\label{rem:imageOM}
\begin{enumerate}[\rm(a)]
\item A more explicit version of the theorem, including an identification of an inverse $\opcomod(C)\to\opmod(A)$, is stated below as Proposition \ref{prop:exp}. 
\item If $A$ has the weak expectation property (e.g., if $A$ is nuclear) then the hypothesis involving $\iota$ can be replaced by the simpler hypothesis that $L$ is faithful as a left $A$-module: see Lemma \ref{lem:WEP-faithful}
\item In view of \eqref{eq:comparison-diagram}, Theorem \ref{thm:WEP} gives the following characterisation of the image of $\functor{L}$: an operator $B$-module $Y$ is completely isometrically isomorphic to one of the form $X\otimes^{\h}_A L$ if and only if $Y$ admits a $C$-comodule structure; and a $B$-linear map $t\in\CB_B(Y_1, Y_2)$ between two such modules is of the form $s\otimes \id_L$ if and only if $t$ is a map of $C$-comodules.
\item The significance of $\functor{L}_C$ being a \emph{completely isometric} equivalence is that, in view of Theorem \ref{thm:Blecher-Morita}, the $\opsp_1$-category $\opcomod(C)$ becomes a complete Morita invariant of the $C^*$-algebra $A$.
\end{enumerate}
\end{remarks} 
 
We shall prove Theorem \ref{thm:WEP} as a consequence of a theorem of Beck (cf.~\cite[Section VI.7]{MacLane}, \cite[Theorem 3.3.14]{ttt}, or \cite[Theorem 4.4.4]{Borceux}). In the following sequence of lemmas we verify that the hypotheses of Beck's theorem are satisfied.

\begin{lemma}\label{lem:K}
The operator space $K$ is completely isometrically isomorphic to an approximately unital operator algebra in such a way that $\eta:A\to K$ is an algebra homomorphism.
\end{lemma}

\begin{proof}
The completely contractive map $\mu:K\times K\to K$ defined by
\[
(l_1\otimes r_1, l_2\otimes r_2)\mapsto l_1\epsilon(r_1\otimes l_2)\otimes r_2 
\]
gives $K$ the structure of an associative algebra, and the commutativity of \eqref{eq:triangle} ensures that the map $\eta:A\to K$ is an algebra homomorphism. The quotient map $K\otimes^{\h} K\to K\otimes^{\h}_A K$ and the counit map $R\otimes^{\h}_A L\to B$ are both completely contractive, and so $\mu$ extends to a completely contractive map $K\otimes^{\h} K\to K$. Since $K$ is nondegenerate as an $A$-bimodule, the image under the contraction $\eta$ of a contractive approximate unit for $A$ will be a contractive approximate unit for $K$. A theorem of Blecher, Ruan and Sinclair \cite[Theorem 2.3.2]{BLM} then implies that $K$ is completely isometrically isomorphic to an operator algebra. 
\end{proof}

\begin{lemma}\label{lem:eta-ci}
Suppose that $A$ is a $C^*$-algebra and that $\eta:A\to K$ is injective. Then for each $X\in \opmod(A)$ the natural transformation $\eta_X:X\to X\otimes^{\h}_A K$ is a complete isometry.
\end{lemma}

\begin{proof}
Lemma \ref{lem:K} implies that there exists a completely isometric algebra embedding of $K$ into a $C^*$-algebra $D$; let us fix such an embedding. The map $\eta:A\to K\into D$ is a  contractive, injective algebra homomorphism of $C^*$-algebras, and hence it is a complete isometry (see e.g.~\cite[A.5.8]{BLM}). Now $\eta_X$ is the composition of the completely isometric isomorphism $X\cong X\otimes^{\h}_A A$ with the map $\id_X\otimes \eta$, which is a complete isometry by Theorem \ref{thm:exact}.
\end{proof}

The following Lemma shows that the hypotheses of Theorem \ref{thm:WEP} can be simplified if $A$ has the weak expectation property:

\begin{lemma}\label{lem:WEP-faithful}
Let $({}_A L_B, {}_B R_A)$ be an adjoint pair of operator bimodules, where $A$ is a $C^*$-algebra with the weak expectation property. The following are equivalent:
\begin{enumerate}[\rm(a)]
\item $L$ is faithful as a left $A$-module.
\item $R$ is faithful as a right $A$-module.
\item The unit map $\eta:A\to K$ is injective.
\item There is a completely contractive $A$-bimodule map $\iota:K\to A^{\dual\dual}$ such that $\iota\circ\eta:A\to A^{\dual\dual}$ is the canonical embedding.
\end{enumerate}
\end{lemma}

\begin{proof}
The commutative diagrams \eqref{eq:triangle} show that the kernel of $\eta$ annihilates the modules $L$ and $R$, and so each of (a) and (b) implies (c). On the other hand, if $a\in A$ annihilates the left $A$-module $L$, then $a$ also annihilates $K=L\otimes^{\h}_B R$ as a left $A$-module, and in particular we have $a\eta(a')=\eta(a)a'=0$ for all $a'\in A$. Since $K$ is nondegenerate as a right $A$-module, we conclude from this last equality that $\eta(a)=0$; thus (c) implies (a). Switching left  and right in the preceding argument shows that (c) implies (b) too. Clearly (d) implies (c). To see that (c) implies (d), choose a completely isometric embedding of $K$ into a $C^*$-algebra $D$ (this is possible by Lemma \ref{lem:K}). Then the composition $A\xrightarrow{\eta} K\to D$ is a completely isometric algebra homomorphism (Lemmas \ref{lem:K} and \ref{lem:eta-ci}), and is therefore an injective $*$-homomorphism (\cite[A.5.8]{BLM} again). Now, since $A$ has the weak expectation property, there is a completely contractive $A$-bimodule map $\iota:D\to A^{\dual\dual}$ whose composition with $\eta:A\to K$ is the canonical embedding.
\end{proof}

To prove Theorem \ref{thm:WEP} we shall exhibit an explicit inverse to $\functor{L}_C$.  The construction is dictated by abstract categorical considerations; see e.g.~\cite[3.2.4]{ttt} for the general construction.

\begin{definition-lemma}\label{lem:RC}
There is a completely contractive functor 
\[
\functor{R}_C : \opcomod^{\cb}(C)\to \opmod(A)
\]
given on objects by 
\[
\functor{R}_C(Z,\delta_Z) = Z\boxtimes_C R   \coloneqq \left \{ \left. \xi \in Z\otimes^{\h}_B R \ \right|\ (\id_Z\otimes\eta_R)(\xi) = (\delta_Z\otimes \id_R)(\xi) \right\}
\] 
and on morphisms by $\functor{R}_C(t) \coloneqq (t\otimes \id_R)\restrict_{Z\boxtimes_C R}$. 
\end{definition-lemma}

\begin{proof}
Note firstly that for each $Z\in \opcomod(C)$  the space $Z\boxtimes_C R$ is a closed $A$-submodule of $Z\otimes^{\h}_B R$, and is therefore an operator $A$-module. The fact that $\functor{R}_C$ is a functor---that is, that $t\otimes\id_R$ maps $Z\boxtimes_C R$ into $W\boxtimes_C R$ for each morphism $t\in \CB_C(Z,W)$---is easily checked as in the algebraic case (see e.g.~\cite[1.10]{Brz-Wis}). The functor $\functor{R}_C$ is completely contractive because $\functor{R}$ is completely contractive, and because restriction of completely bounded maps to a closed subspace is likewise completely contractive. 
\end{proof}
 
We shall prove the following more precise version of Theorem \ref{thm:WEP}. To understand the statement, note that that for each $X\in \opmod(A)$ the $A$-module $\functor{R}_C\functor{L}_C (X) = ( X\otimes^{\h}_A L )\boxtimes_C R$ is a closed submodule of $X\otimes^{\h}_A K = X\otimes^{\h}_A L\otimes^{\h}_A R$; while for each $(Z,\delta_Z)\in \opcomod^{\cb}(C)$ the $C$-comodule $\functor{L}_C\functor{R}_C(Z)=(Z\boxtimes_C R)\otimes^{\h}_A L$ can be identified with a closed $B$-submodule of $Z\otimes^{\h}_B C= Z\otimes^{\h}_B R\otimes^{\h}_A L$, via the completely isometric embedding 
\[
 (Z\boxtimes_C R)\otimes^{\h}_A L \xrightarrow{ \textrm{inclusion}\otimes\id_L} Z \otimes^{\h}_B R\otimes^{\h}_A L.
\]

\begin{proposition}\label{prop:exp}
\begin{enumerate}[\rm(a)]
\item   Suppose that $A$ is a $C^*$-algebra, and that $\eta:A\to K$ is injective. For each $X\in \opmod(A)$ the map $\eta_X:X\to X\otimes^{\h}_A K$ restricts to a completely isometric isomorphism of operator $A$-modules $\eta:X\to \functor{R}_C\functor{L}_C(X)$.
\item Suppose that $A$ is a $C^*$-algebra and that there exists a completely contractive $A$-bimodule map $\iota:K\to A^{\dual\dual}$ such that $\iota\circ \eta:A\to A^{\dual\dual}$ is the canonical embedding. Then for each $(Z,\delta_Z)\in \opcomod^{\cb}(C)$ the map $\delta_Z:Z\to Z\otimes^{\h}_B C$ restricts to a completely bounded natural isomorphism of operator $C$-comodules
$
\delta_Z: Z\to \functor{L}_C\functor{R}_C(Z)
$, which is completely isometric if $(Z,\delta_Z)\in \opcomod(C)$. In particular, $\functor{L}_C:\opmod(A)\to \opcomod(C)$ and $\functor{R}_C:\opcomod(C)\to\opmod(A)$ are mutually inverse completely isometric equivalences.
\end{enumerate}
\end{proposition}

The technical heart of the proof of Proposition \ref{prop:exp} is contained in the following two lemmas. 

\begin{lemma}\label{lem:sigma}
Assume the hypotheses of Theorem \ref{thm:WEP}. For each $X\in \opmod(A)$ there is a completely bounded map $\sigma_X:X^\dual\to (X\otimes^{\h}_A K)^\dual$ which satisfies $\eta_X^\dual\circ \sigma_X =\id_{X^\dual}$, and which is natural in the sense that for every morphism $t\in \CB_A(X,Y)$ in $\opmod(A)$ one has $\sigma_X\circ t^\dual = (t\otimes \id_K)^\dual\circ \sigma_Y$.
\end{lemma}

\begin{proof}
As explained in \cite[3.8.9]{BLM}, the operator $A$-module structure on $X$ extends to a compatible operator $A^{\dual\dual}$-module structure on the bidual operator space $X^{\dual\dual}$. Consider the completely contractive, right-$A$-linear map 
$\tau_X:X\otimes^{\h}_A K\to X^{\dual\dual}$
defined as the composition
\[
X\otimes^{\h}_A K \xrightarrow{[X\into X^{\dual\dual}] \otimes \iota} X^{\dual\dual} \otimes^{\h}_A A^{\dual\dual} \xrightarrow{\act} X^{\dual\dual}.
\]
Note that for $x\in X$ and $a\in A$ we have 
\[
\tau_X\circ \eta_X(xa)=\tau_X(x\otimes \eta(a)) = x\cdot\iota\circ\eta(a) =xa,
\]
showing that $\tau_X\circ\eta_X$ is the canonical embedding $X\into X^{\dual\dual}$. Dualising $\tau_X$, and restricting to the image of the natural embedding $X^\dual\into X^{\dual\dual\dual}$, we obtain a completely contractive  map 
\[
\sigma_X\coloneqq\tau_X^\dual\restrict_{X^\dual}:X^\dual\to (X\otimes^{\h}_A K)^\dual.
\]
Since $\tau_X\circ \eta_X$ is the embedding $X\into X^{\dual\dual}$, it follows by duality that $\eta_X^\dual\circ\sigma_X=\id_{X^\dual}$. Since $\sigma_X$ and $\eta_X^\dual$ are both completely contractive, this last equality implies that $\sigma_X$ is a complete isometry.

To verify the asserted naturality,  note that if $t\in \CB_A(X,Y)$ then the bidual map $t^{\dual\dual}:X^{\dual\dual}\to Y^{\dual\dual}$ is $A^{\dual\dual}$-linear: this follows  from the separate weak-$*$ continuity of the bidual module action (see \cite[3.8.9]{BLM}). We thus have  for every $k\in K$ and $x\in X$ that
\[
t^{\dual\dual}\circ\tau_X(x\otimes k) = t^{\dual\dual}(x\iota(k))=t(x)\iota(k)=\tau_Y\circ (t\otimes \id_K) (x\otimes k),
\]
from which it follows by duality that $\sigma_X\circ t^\dual = (t\otimes \id_K)^\dual\circ \sigma_Y$ as required.
\end{proof}

\begin{lemma}\label{lem:WEP-heart}
Assume the hypotheses of Theorem \ref{thm:WEP}, and suppose that $f,g:X\to Y$ is a pair of morphisms in $\opmod(A)$ such that there exists a morphism $\lambda: Y\otimes^{\h}_A L \to  X\otimes^{\h}_A L$ in $\opmod(B)$ satisfying $\lambda\circ (f\otimes \id_L)=\id_{X\otimes^{\h}_A L}$ and $(h\otimes \id_L)\circ \lambda \circ (g\otimes \id_L)=0$, where $h\coloneqq f-g$. Then the map $h:X\to Y$ has closed range, and induces a completely bounded isomorphism $X/\ker h \to \image(h)$.
\end{lemma}
 
\begin{proof} 
Consider the map $s:X^\dual\to Y^\dual$ defined as the composition $s \coloneqq \eta_Y^\dual\circ (\lambda\otimes \id_R)^\dual \circ \sigma_X$. 
Noting that the given relations between $f$, $g$, $h$, and $\lambda$ immediately give 
\[
(h\otimes \id_K)\circ (\lambda\otimes \id_R) \circ (h\otimes \id_K) = h\otimes \id_K,
\]
we compute using the naturality properties of $\eta$ and $\sigma$ and find that 
\[
\begin{aligned}
 h^\dual \circ s \circ h^\dual 
 & = h^\dual \circ \eta_Y^\dual\circ (\lambda\otimes \id_R)^\dual \circ \sigma_X \circ h^\dual \\
 & = \eta_X^\dual \circ (h\otimes \id_K)^\dual \circ (\lambda\otimes \id_R)^\dual \circ (h\otimes \id_K)^\dual \otimes \sigma_Y \\
 & = \eta_X^\dual \circ \sigma_X \circ h^\dual = h^\dual.
 \end{aligned}
\]
 
  The rest of the proof is an application of basic duality theory. The claim is that if $X$ and $Y$ are operator spaces, and if $h:X\to Y$ and $s:X^\dual \to Y^\dual$ are completely bounded maps satisfying $h^\dual   s  h^\dual=h^\dual$, then $h$ has closed range and the induced map 
\[
\widetilde{h}:X/\ker(h)\to \image(h),\qquad x+\ker(h)\mapsto h(x)
\]
is an isomorphism. 
The equality $h^\dual   s h^\dual=h^\dual$ implies that $h^\dual  s$ is an idempotent on $X^\dual$ with the same image as $h^\dual$, and so $h^\dual$ has closed range. It now follows from standard Banach space theory that $h$ also has closed range. Moreover, the map
\[
\widetilde{h^\dual}:Y^\dual/\ker(h^\dual) \to \image(h^\dual),\qquad y^\dual+\ker(h^\dual)\mapsto h^\dual(y^\dual)
\]
is a completely bounded isomorphism, with inverse $x^\dual\mapsto \sigma(x^\dual)+\ker(h^\dual)$.

The standard duality relations between subspaces, quotients, images and kernels (cf.~\cite[1.4.4]{BLM}) give a commuting diagram
 \[
\xymatrix@C=40pt{
\image(h)^\dual \ar[r]^-{ (\widetilde{h} )^\dual} \ar[d]_-{\cong} & (X/\ker(h))^\dual \ar[d]^-{\cong} \\
Y^\dual/\ker(h^\dual) \ar[r]^-{\widetilde{h^\dual}} & \image(h^\dual)
}
\]
from which we conclude that $ (\widetilde{h} )^\dual$ is, like $\widetilde{h^\dual}$, an isomorphism. Thus $\widetilde{h}$ is an isomorphism as well (cf.~\cite[1.4.3]{BLM}).
\end{proof}

\begin{corollary}\label{cor:contractible-equalisers}
Assume the hypotheses of Theorem \ref{thm:WEP}. If $f$, $g$ and $h=f-g$ are as in Lemma \ref{lem:WEP-heart}, then the map
\[
\ker(h)\otimes^{\h}_A L \xrightarrow{\operatorname{inclusion}\otimes \id_L} X\otimes^{\h}_A L
\]
is a complete isometry onto $\ker(h\otimes \id_L)$.
\end{corollary}
 
\begin{proof}
Lemma \ref{lem:WEP-heart} implies that the restricted map $h:X\to \image(h)$ is conjugate, by a completely bounded isomorphism, to the quotient map $X\to X/\ker h$, and so the corollary follows from Theorem \ref{thm:exact}.
\end{proof}

The proof of Proposition \ref{prop:exp}  follows from Lemma \ref{lem:eta-ci} and Corollary \ref{cor:contractible-equalisers} by an argument due to Beck (cf.~\cite[Section VI.7]{MacLane}, \cite[Theorem 3.3.14]{ttt}, or \cite[Theorem 4.4.4]{Borceux}), which applies in a very general context and whose statement and proof are typically couched in correspondingly general categorical language. For the benefit of the reader having only a slight acquaintance with (or tolerance for) category theory, let us give a self-contained presentation of Beck's argument in the present  setting.   

\begin{proof}[Proof of Proposition \ref{prop:exp}(a)]
The map $\eta_X:X\to X\otimes^{\h}_A K$ is a complete isometry by Lemma \ref{lem:eta-ci}, and its image is contained in $ \functor{R}_C  \functor{L}_C (X)$ by the naturality property \eqref{eq:natural} of $\eta$:
\[
\begin{aligned}
&(\delta_{\functor{L}_C(X)}\otimes \id_R)\circ \eta_X = (\eta_X\otimes \id_K) \circ \eta_X = \eta_{X\otimes^{\h}_A K}\circ \eta_X\\
& = (\id_X\otimes \eta_K)\circ \eta_X = (\id_{\functor{L}_C(X)}\otimes \eta_R)\circ\eta_X.
\end{aligned}
\]

To prove the reverse inclusion $ \functor{R}_C \functor{L}_C (X)\subseteq\image(\eta_X)$, we consider the quotient map $q:K\to K/A$; this is a map of operator $A$-bimodules. Theorem \ref{thm:exact}, modified in the obvious way so that the roles of left and right modules are reversed, implies that the image of $\eta_X=\id_X\otimes \eta$ is equal to the kernel of $\id_X\otimes q$. We have
\[
\begin{aligned}
(\id_X\otimes \eta_{K/A})\circ(\id_X\otimes q)\restrict_{\functor{R}_C\functor{L}_C( X)} & = 
(\id_X\otimes q\otimes \id_K)\circ (\id_X\otimes \eta_K)\restrict_{\functor{R}_C\functor{L}_C (X)}\\
&= (\id_X\otimes q\otimes \id_K)\circ (\eta_X\otimes \id_K)\restrict_{\functor{R}_C\functor{L}_C (X)}\\
&=0,
\end{aligned}
\]
where the first equality holds by the naturality of $\eta$, the second holds by the definition of $\functor{R}_C\functor{L}_C( X)$, and the third holds because $q\circ\eta =0$. Since $\id_X\otimes \eta_{K/A}$ is a complete isometry, by Lemma \ref{lem:eta-ci}, we conclude from the above equality that $\functor{R}_C\functor{L}_C (X)\subseteq \ker (\id_X\otimes q)=\image(\eta_X)$ as required.
\end{proof}

\begin{proof}[Proof of Proposition \ref{prop:exp}(b)]
Consider a comodule $(Z,\delta_Z)\in \opcomod^{\cb}(C)$. The map $\delta_Z$ is natural in $Z$, by virtue of the defining property \eqref{eq:comodule-morphism} of the morphisms in $\opcomod(C)$. We observed in Remarks \ref{rem:coalgebras-comodules} that $\delta_Z$ is a complete embedding, and a complete isometry if $(C,\delta_Z)\in \opcomod(C)$. 

We are going to apply Lemma \ref{lem:WEP-heart} to the $A$-modules $X=Z\otimes^{\h}_B R$ and $Y=Z\otimes^{\h}_B C\otimes^{\h}_B R$, and the maps 
\[
\xymatrix@C=60pt{
 Z\otimes^{\h}_B R  \ar@<.5ex>[r]^-{f = \id_Z\otimes\eta_R} \ar@<-.5ex>[r]_-{g=\delta_Z\otimes\id_R} & Z\otimes^{\h}_B C\otimes^{\h}_B R,
 }
\quad
\xymatrix@C=60pt{Z\otimes^{\h}_B C \otimes^{\h}_B C \ar[r]^-{\lambda = \id_Z\otimes \epsilon_C} & Z\otimes^{\h}_B C } 
\] 
and $h=f-g$. We have
\[
\lambda\circ(f\otimes \id_L) = (\id_Z\otimes \epsilon_C)\circ (\id_Z\otimes \delta) = \id_{Z\otimes^{\h}_B C} 
\]
by the coalgebra axioms (Definition \ref{def:coalgebra}), while
\[
\begin{aligned}
(h\otimes \id_L)\circ\lambda\circ (g\otimes \id_L) & = (\id_Z\otimes \delta - \delta_Z\otimes\id_C)\circ (\id_Z\otimes \epsilon_C) \circ (\delta_Z\otimes \id_C) \\
& = (\id_Z\otimes \delta - \delta_Z\otimes\id_C)\circ \delta_Z \circ \epsilon_Z = 0,
\end{aligned}  
\]
where the second equality holds by virtue of the naturality \eqref{eq:natural} of $\epsilon$, and the third holds by \eqref{eq:comodule-axioms}. Thus Lemma \ref{lem:WEP-heart} and Corollary \ref{cor:contractible-equalisers} apply, giving an equality of $B$-modules
\[
\functor{L}_C\functor{R}_C(Z) \coloneqq \ker(h)\otimes^{\h}_A L = \ker (h\otimes \id_L).
\]

Now the coassociativity property \eqref{eq:comodule-axioms} shows as above that $\image(\delta_Z)\subseteq \ker(h\otimes \id_L)=\functor{L}_C\functor{R}_C(Z)$. The reverse inclusion $\ker(h\otimes \id_L)\subseteq \image(\delta_Z)$ follows from the counitality of $\delta$ and the naturality of $\epsilon$, thus:
\[
\begin{aligned}
\id_{Z\otimes^{\h}_B C} -\delta_Z\circ \epsilon_Z & = (\id_Z\otimes \epsilon_C)\circ(\id_Z\otimes \delta) - (\id_Z\otimes \epsilon_C)\circ (\delta_Z\otimes\id_C) \\
& = (\id_Z\otimes \epsilon_C)\circ (h\otimes\id_L).
\end{aligned}
\]
We   conclude that $\delta_Z$ is a completely bounded (or if $(Z,\delta_Z)\in \opcomod(C)$, completely isometric) natural isomorphism of $B$-modules, from $Z$ to $\functor{L}_C \functor{R}_C(Z)$.

It remains to show that $\delta_Z$ is an isomorphism of $C$-comodules, which is easy: \eqref{eq:comodule-axioms}   ensures that $\delta_Z$ is a map of comodules, and then the usual algebraic argument shows that the inverse $\delta_Z^{-1}$ of the $B$-module isomorphism $\delta_Z$ is a map of comodules too.
\end{proof}

\begin{remark}\label{rem:bimodules}
Fix a third   operator algebra $D$, and consider the $\opsp_1$-categories $\opmod(D,A)$ and $\opmod(D,B)$ of operator $D$-$A$ bimodules and operator $D$-$B$ bimodules, with morphisms the completely bounded bimodule maps. If $({}_A L_B, {}_B R_A)$ is an adjoint pair of operator bimodules, and if $C=R\otimes^{\h}_A L$ is the associated coalgebra, then we consider the $\opsp_1$-category $\opcomod(D,C)$ whose objects are pairs $(Z,\delta_Z)$ consisting of an operator $D$-$B$ bimodule $Z$ and a completely contractive $D$-$B$-bimodule map $\delta_Z:Z\to Z\otimes^{\h}_B C$ making the diagrams \eqref{eq:comodule-axioms} commute; the morphisms in $\opcomod(D,C)$ are the completely bounded $D$-$B$-bimodule maps making \eqref{eq:comodule-morphism} commute. Just as in the case $D=\C$ we have a comparison functor
\begin{equation}\label{eq:bimodule-comparison}
\opmod(D,A)\to \opcomod(D,C),\qquad X\mapsto (X\otimes^{\h}_A L, \eta_X\otimes \id_L),
\end{equation}
and the proof of Theorem \ref{thm:WEP} carries over verbatim to this setting: if $A$ is a $C^*$-algebra and there is a map $\iota:K\to A^{\dual\dual}$ as in Theorem \ref{thm:WEP}, then the functor \eqref{eq:bimodule-comparison} is a completely isometric equivalence.
\end{remark}
  
\section{Descent of Hilbert ${C^*}$-modules}\label{sec:descent-Cmod}

In this section we return to the setting of Example \ref{ex:corresp}: we consider the adjoint pair $({}_A F_B, {}_B F^\adjoint_A, )$ associated to a $C^*$-correspondence ${}_A F_B$ between $C^*$-algebras $A$ and $B$, where $A$ acts on $F$ through a nondegenerate $*$-homomorphism into the $C^*$-algebra $\Compact_B(F)\cong K$ of $B$-compact operators on $F$.  We are going to prove a $C^*$-module version of Theorem \ref{thm:WEP} for this adjoint pair.

We thus consider for each $C^*$-algebra $A$ the category $\cmod(A)$ of right Hilbert $C^*$-modules over $A$, with adjointable $A$-module maps (notation: $\CB_A^*(X,Y)$) as morphisms. We refer to \cite{Lance} or \cite[Chapter 8]{BLM} for background on Hilbert $C^*$-modules; see in particular \cite[8.2.1]{BLM} for the canonical operator-module structure on a Hilbert $C^*$-module, and \cite[8.2.11]{BLM} for the completely isometric identification between the Haagerup tensor product and the Hilbert $C^*$-module tensor product. 

The operator $B$-coalgebra $C=F^\adjoint\otimes^{\h}_A F$ associated to the adjoint pair $( F, F^\adjoint)$ carries a completely isometric, conjugate-linear involution 
\[
*:C\to C, \qquad (f_1^\adjoint\otimes f_2)^*\coloneqq f_2^\adjoint\otimes f_1.
\]
This involution satisfies $(b_1 c b_2)^* = b_2^* c^* b_1^*$ for all $b_1,b_2\in B$ and $c\in C$. The involution   is also compatible with the coalgebra structure: indeed, we have a second completely isometric, conjugate-linear involution
\[
*:C\otimes^{\h}_B C\to C\otimes^{\h}_B C,\qquad (c_1\otimes c_2)^* \coloneqq c_2^*\otimes c_1^*,
\]
and one has $\delta(c)^*=\delta(c^*)$ and $\epsilon(c)^*=\epsilon(c^*)$ for all $c\in C$.

\begin{definition}\label{def:star-comodule}
A  \emph{Hilbert $C^*$-comodule} over $C$ is a right Hilbert $C^*$-module $Z$ over $B$,  equipped with a completely contractive $B$-linear map $\delta_Z:Z\to Z\otimes^{\h}_B C$ making \eqref{eq:comodule-axioms} commute, such that  
\begin{equation}\label{eq:coaction-hermitian}
 \langle z_1 | \delta_Z(z_2) \rangle_C = \langle z_2 | \delta_Z(z_1)\rangle_C^* 
\end{equation}
for all $z_1,z_2\in Z$. Here we are applying the pairing 
\[
  \langle\cdot |\cdot\rangle_C: Z\times (Z\otimes^{\h}_B C) \to C,\qquad \langle z_1 | z_2\otimes c\rangle_C\coloneqq \langle z_1|z_2\rangle c.
\]
A morphism of Hilbert $C^*$-comodules $(Z,\delta_Z)\to (W,\delta_W)$ is an adjointable map of Hilbert $C^*$-modules $t:Z\to W$ making \eqref{eq:comodule-morphism} commute. We denote by $\ccomod(C)$ the category of right Hilbert $C^*$-comodules over $C$, and by $\CB_C^*(Z,W)$ the space of morphisms in $\ccomod(C)$.
\end{definition}

We shall also  consider the category $\ccomod^{\cb}(C)$ of \emph{completely bounded} Hilbert $C^*$-comodules: that is, pairs $(Z,\delta_Z)$ which are as in Definition \ref{def:star-comodule}, except that the coaction $\delta_Z$ is allowed to be merely  completely bounded, rather than completely contractive. We will later show that in the situations of interest to us, $\delta_Z$ is in fact automatically completely contractive: see Corollary \ref{cor:coaction-cb-cc}.

\begin{lemma}\label{lem:Cstar-cat}
If $t:(Z,\delta_Z)\to (W,\delta_W)$ is a morphism in $\ccomod^{\cb}(C)$, then the adjoint map $t^*:W\to Z$ is also a morphism in $\ccomod^{\cb}(C)$.
\end{lemma}

\begin{proof}
For each $z\in Z$ and $w\in W$ we have 
\[
\begin{aligned}
& \langle z | \delta_Z(t^*w)\rangle_C = \langle t^*w | \delta_Z(z)\rangle^*_C = \langle w | (t\otimes\id_C)\delta_Z(z)\rangle^*_C \\
&= \langle w | \delta_W(tz)\rangle^*_C = \langle tz | \delta_W(w)\rangle_C = \langle z | (t^*\otimes\id_C)\delta_W(w)\rangle_C,
\end{aligned}
\]
where the first and the fourth equalities hold because of the Hermitian property \eqref{eq:coaction-hermitian} of $\delta_Z$ and $\delta_W$; the second and the fifth equalities hold because $t$ and $t^*$ are adjoints of one another; and the third equality holds because $t$ was assumed to be a comodule map. 

Now, Blecher has shown (\cite[Theorem 3.10]{Blecher-newapproach}, cf.~\cite[Corollary 8.2.15]{BLM}) that as $z$ ranges over $Z$, the maps 
\[
\langle z|: Z\otimes^{\h}_B C \to C,\qquad \xi \mapsto \langle z|\xi\rangle_C 
\]
separate the points of $Z\otimes^{\h}_B C$. It follows from this, and from the above computation, that $\delta_Z(t^*w) = ( t^*\otimes \id_C)\delta_W(w)$ for all $w\in W$, and so $t^*$ is indeed a map of comodules.
\end{proof}

The category $\cmod(B)$ of Hilbert $C^*$-modules over $B$ is a \emph{$C^*$-category}: each of its morphism spaces $\CB^*_B(X,Y)$ carries a Banach-space norm and a  conjugate-linear map $*:\CB^*_B(X,Y)\to \CB^*_B(Y,X)$ satisfying the natural `multi-object' analogues of the axioms of a $C^*$-algebra; and for each $t\in \CB^*_B(X,Y)$ we have $t^*t\geq 0$ in the $C^*$-algebra $\CB^*_B(X,X)$. See \cite{GLR} for more on the notion of a $C^*$-category. The morphism spaces $\CB_C^*(X,Y)$ in $\ccomod^{\cb}(C)$ are obviously closed subspaces of the morphism spaces $\CB_B^*(X,Y)$ in $\cmod(B)$, and Lemma \ref{lem:Cstar-cat} shows that the involution $*$ on $\CB_B^*(\argument,\argument)$ restricts to an involution on $\CB_C^*(\argument,\argument)$. 
Thus: 

\begin{corollary}\label{cor:Cstar-cat}
The categories $\ccomod(C)$ and $\ccomod^{\cb}(C)$  
are $C^*$-categories.
\hfill\qed
\end{corollary}

For each right Hilbert $C^*$-module $X$ over $A$, the Haagerup tensor product $\functor{F}(X)\coloneqq X\otimes^{\h}_A F$ is a right Hilbert $C^*$-module over $B$, with inner product
\[
\langle x_1\otimes f_1 | x_2\otimes f_2 \rangle = \big\langle f_1\, \big|\, \langle x_1|x_2\rangle f_2 \big\rangle.
\]
We consider, as in Example \ref{ex:comodule}, the completely contractive coaction 
\[
\delta_{\functor{F}(X)}\coloneqq \eta_X\otimes \id_F: X\otimes^{\h}_A F\to X\otimes^{\h}_A F\otimes^{\h}_B C.
\]
 
\begin{lemma} For each Hilbert $C^*$-module $X$ over $A$, the pair $(\functor{F}(X),\delta_{\functor{F}(X)})$ is a Hilbert $C^*$-comodule over $C$.
\end{lemma}

\begin{proof}
We must show that the coaction $\delta_{\functor{F}(X)}$ satisfies the Hermitian property \eqref{eq:coaction-hermitian}; i.e., that 
\[
\langle x_1\otimes a_1 f_1 | x_2\otimes \eta(a_2)\otimes f_2\rangle_C = \langle x_2\otimes a_2f_2 | x_1\otimes \eta(a_1)\otimes f_1\rangle_C^*
\] 
for all $f_1,f_2\in F$, $a_1,a_2\in A$ and $x_1,x_2\in X$.

We claim that for each $k\in \Compact_B(F)\cong F\otimes^{\h}_B F^\adjoint$   one has
\[
\langle x_1\otimes f_1 | x_2\otimes k\otimes f_2\rangle_C = \left( k^* \langle x_2|x_1\rangle f_1 \right)^\adjoint \otimes f_2.
\]
Indeed, it suffices to check this   for a `rank-one' operator $k=f_2\otimes f_3^\adjoint$,   which is done by a straightforward computation. It follows from this that 
\[
\begin{aligned}
& \langle x_1\otimes a_1 f_1 | x_2\otimes \eta(a_2)\otimes f_2\rangle_C 
= \left( \eta(a_2^*) \langle x_2|x_1\rangle a_1f_1 \right)^\adjoint\otimes f_2 \\
& = \left( \langle x_2a_2 | x_1 a_1\rangle f_1\right)^\adjoint \otimes f_2  
= f_1^\adjoint \otimes \langle x_1 a_1 | x_2 a_2\rangle f_2 \\
& = \left( \left( \eta(a_1^*) \langle x_1| x_2\rangle a_2 f_2\right)^\adjoint \otimes f_1 \right)^* = \langle x_2\otimes a_2f_2 | x_1\otimes \eta(a_1)\otimes f_1\rangle_C^*
\end{aligned}
\]
as required.
\end{proof}

Setting $\functor{F}_C(t)\coloneqq t\otimes \id_C$ for each morphism $t\in \cmod(A)$, we obtain a $*$-functor 
\[
\functor{F}_C:\cmod(A)\to \ccomod(C).
\]
We shall prove the following as a corollary to Theorem \ref{thm:WEP}:

\begin{theorem}\label{thm:Hermitian}
Let ${}_A F_B$ be a $C^*$-correspondence between $C^*$-algebras $A$ and $B$. Assume that $A$ acts faithfully by $B$-compact operators on $F$, and that the inclusion $\eta:A\to \Compact_B(F)$ admits a weak expectation $\Compact_B(F)\to A^{\dual\dual}$. Then the  functor 
\[
\functor{F}_C:\cmod(A)\to \ccomod(C),\qquad X\mapsto (X\otimes^{\h}_A F, \eta_X \otimes \id_F)
\]
is a unitary equivalence of $C^*$-categories, from the  category of right Hilbert $C^*$-modules over $A$   to the  category of right Hilbert $C^*$-comodules over $C$.   
\end{theorem}

\begin{remark}\label{rem:imageCM}
We note, as in Remark \ref{rem:imageOM}(c), that Theorem \ref{thm:Hermitian} gives a description of the image of the functor $\functor{F}$: a Hilbert $C^*$-module $Y$ over $B$ is unitarily isomorphic to one of the form $X\otimes_A F$ for some Hilbert $C^*$-$A$-module $X$, if and only if $Y$ can be equipped with the structure of a Hilbert $C^*$-comodule over $C$. Similar considerations apply to morphisms.
\end{remark}
  
To construct an inverse to $\functor{F}_C$ we consider, for each (completely bounded) Hilbert $C^*$-comodule $(Z,\delta_Z)\in \ccomod^{\cb}(C)$, the right operator $A$-module $Z\boxtimes_C F^\adjoint$ defined as in Definition/Lemma \ref{lem:RC}:
\[
Z\boxtimes_C F^\adjoint \coloneqq \left\{\left. \xi\in Z\otimes^{\h}_B F^\adjoint \ \right|\ (\id_Z\otimes \eta_{F^\adjoint})(\xi) = (\delta_Z\otimes \id_{F^\adjoint})(\xi)\right\}.
\]
The module $F^\adjoint$ is right Hilbert $C^*$-module over the $C^*$-algebra $K=\Compact_B(F)\cong F\otimes^{\h}_B F^\adjoint$ of $B$-compact operators on $F$: the $K$-valued inner product is given by 
$\langle f_1^\adjoint | f_2^\adjoint\rangle \coloneqq f_1\otimes f_2^\adjoint$.
The left action of $B$ on $F^\adjoint$ is via a $*$-homomorphism into the $C^*$-algebra of $K$-adjointable operators, and it follows from this that the Haagerup tensor product $Z\otimes^{\h}_B F^\adjoint$ is a right Hilbert $C^*$-module over $K$, with inner product
\begin{equation}\label{eq:K-valued-ip}
\langle z_1\otimes f_1^\adjoint | z_2\otimes f_2^\adjoint\rangle \coloneqq f_1\langle z_1|z_2\rangle \otimes f_2^\adjoint.
\end{equation}

\begin{lemma}\label{lem:boxtimes-ip}
Suppose that the action homomorphism $\eta:A\to K$ is injective. For each Hilbert $C^*$-comodule $(Z,\delta_Z)\in \ccomod^{\cb}(C)$ one has $\langle \xi|\zeta\rangle \in \eta(A)\subseteq K$ for all $\xi,\zeta\in Z\boxtimes_C F^\adjoint$. Consequently, $Z\boxtimes_C F^\adjoint$ is a Hilbert $C^*$-module over $A$, with inner product $\langle \xi|\zeta\rangle_A\coloneqq \eta^{-1}(\langle \xi|\zeta\rangle_K)$.
\end{lemma}

\begin{proof}
We  claim that
\begin{equation}\label{eq:eta-A-image}
\eta(A) = \left\{k\in K\ \left|\ f_1^\adjoint\otimes k(f_2)\  = (k^*(f_1))^\adjoint\otimes f_2 \textrm{ in $C$, for all }f_1,f_2\in F\right.\right\}.
\end{equation}
To prove this  we first note that Proposition \ref{prop:exp}(a), applied to $X=A$, shows that 
\begin{equation}\label{eq:image-eta}
\eta(A) = \left\{ a_1k a_2\in K\ \left|\ \eta(a_1)\otimes ka_2  = a_1k\otimes \eta(a_2) \textrm{ in }K\otimes^{\h}_A K\right.\right\}.
\end{equation}
Now  \cite[Theorem 3.10]{Blecher-newapproach} implies that as $f_1$ and $f_2$ range over $F$, the maps 
\[
K\otimes^{\h}_A K\to C,\qquad k_1\otimes k_2\mapsto (k_1^*(f_1))^\adjoint\otimes k_2(f_2)
\]
separate the points of $K\otimes^{\h}_A K$, and the identity \eqref{eq:eta-A-image}  follows from this fact combined with \eqref{eq:image-eta}.

We shall use the pairings
\[
\begin{aligned}
& \langle\cdot | \cdot\rangle_{F} : Z\times Z\otimes^{\h}_B F^\adjoint \to F  , \qquad \langle  z_1 | z_2\otimes f^\adjoint \rangle_{F} \coloneqq f\langle z_2 |z_1\rangle    \qquad \textrm{and}\\
&  \langle \cdot | \cdot\rangle_Z : Z\otimes^{\h}_B F^\adjoint \times F  \to Z,\qquad  \langle z\otimes f_1^\adjoint | f_2\rangle_Z  \coloneqq z\langle f_1|f_2\rangle,
\end{aligned}
\]
which are related to the $K$-valued inner product \eqref{eq:K-valued-ip} by the equality 
\begin{equation}\label{eq:ip-comparison}
\langle \xi | \zeta\rangle (f) =  \big\langle \langle \zeta | f\rangle_Z  \big|   \xi\big\rangle_F 
\end{equation}
for $\xi,\zeta\in Z\otimes^{\h}_B F^\adjoint$ and $f\in F$.

For each $f\in F$ and $z\in Z$ consider the map 
\begin{equation}\label{eq:langle-f-z-rangle}
\langle z|\cdot|f\rangle : Z\otimes^{\h}_B C\otimes^{\h}_B F^\adjoint \to C,\qquad z_1\otimes c\otimes f_1^\adjoint \mapsto \langle z|z_1\rangle c \langle f_1|f_2\rangle.
\end{equation}
It is easily checked that the composition
\[
Z\otimes^{\h}_B F^\adjoint \xrightarrow{\id_Z\otimes \eta_{F^\adjoint}} Z\otimes^{\h}_B C\otimes^{\h}_B F^\adjoint \xrightarrow{\langle z|\cdot|f\rangle} C
\]
is given by $\xi\mapsto \langle z|\xi\rangle_F^\adjoint \otimes f$, while the composition
\[
Z\otimes^{\h}_B F^\adjoint \xrightarrow{\delta_Z\otimes \id_{F^\adjoint}} Z\otimes^{\h}_B C\otimes^{\h}_B F^\adjoint \xrightarrow{\langle z|\cdot|f\rangle} C
\]
is given by $\xi\mapsto  \big\langle z\ \big|\ \delta_Z
\left( \langle \xi|f\rangle_Z \right)  \big\rangle_C$. Consulting the definition of $Z\boxtimes_C F^\adjoint$, we find that
\begin{equation}\label{eq:boxtimes-ident}
 \langle z|\xi\rangle_F^\adjoint \otimes f = \big\langle z\ \big|\ \delta_Z
\left( \langle \xi|f\rangle_Z \right)  \big\rangle_C \quad \textrm{for all}\quad f\in F,\ z\in Z,\ \xi\in Z\boxtimes_C F^\adjoint.
\end{equation} 
 
Now  applying  \eqref{eq:ip-comparison} and \eqref{eq:boxtimes-ident}, we find for $\xi,\zeta\in Z\boxtimes_C F^\adjoint$ and $f_1,f_2\in F$ that
\[
\begin{aligned}
& \left( \langle \xi |\zeta\rangle^*(f_1)\right)^\adjoint\otimes f_2  = \left( f_2^\adjoint \otimes \langle \zeta|\xi\rangle f_1\right)^*  = \left( f_2^\adjoint \otimes \big\langle \langle\xi|f_1\rangle_Z\ \big|\ \zeta \big\rangle \right)^*\\ &  = \big\langle \langle \xi|f_1\rangle_Z \ \big|\ \zeta\big\rangle_F^\adjoint \otimes f_2  = \big\langle \langle \xi|f_1\rangle_Z \ \big|\ \delta_Z\left(\langle \zeta|f_2\rangle_Z\right)\big\rangle_C.
\end{aligned}
\]
The Hermitian property  \eqref{eq:coaction-hermitian} of $\delta_Z$ translates, through the above chain of equalities, to the equality
\[
\left( \langle \xi|\zeta\rangle^* f_1\right)^\adjoint \otimes f_2 = \left( \left(\langle \zeta|\xi\rangle^* f_2\right)^\adjoint \otimes f_1\right)^* = f_1^\adjoint \otimes \langle \xi|\zeta\rangle f_2.
\]
In view of  \eqref{eq:eta-A-image}, this computation shows that that $\langle \xi |\zeta\rangle\in \eta(A)$ for all $\xi,\zeta\in Z\boxtimes_C F^\adjoint$, and so the latter module is indeed a Hilbert $C^*$-module over $A$. 
\end{proof}

\begin{lemma}\label{lem:FCstar}
Suppose that $\eta:A\to K$ is injective. The assignment $(Z,\delta_Z)\mapsto Z\boxtimes_C F^\adjoint$ extends to a $*$-functor $\functor{F}_C^\adjoint : \ccomod^{\cb}(C)\to \cmod(A)$, defined on morphisms by $\functor{F}_C^\adjoint(t)\coloneqq (t\otimes \id_{F^\adjoint})\restrict_{Z\boxtimes_C F^\adjoint}$.
\end{lemma}

\begin{proof}
If $t:(Z,\delta_Z)\to (W,\delta_W)$ is an adjointable map of $C^*$-comodules, then  the maps 
\[
t\otimes \id_{F^\adjoint} : Z\otimes^{\h}_B F^\adjoint \to W \otimes^{\h}_B F^\adjoint \quad \textrm{and}\quad t^*\otimes \id_{F^\adjoint}: W \otimes^{\h}_B F^\adjoint \to Z\otimes^{\h}_B F^\adjoint
\] 
are mutually adjoint maps of Hilbert $C^*$-modules over $K$. Since both $t$ and $t^*$ are comodule maps (Lemma \ref{lem:Cstar-cat}),  $t\otimes \id_{F^\adjoint}$ and $t^*\otimes \id_{F^\adjoint}$ restrict to maps between the closed $A$-submodules $Z\boxtimes_C F^\adjoint$ and $W\boxtimes_C F^\adjoint$, and we have by the definition of the $A$-valued inner products on $Z\boxtimes_C F^\adjoint$ and $W\boxtimes_C F^\adjoint$ that 
\[
\langle \functor{F}_C^\adjoint(t)\xi |\zeta\rangle = \eta^{-1}\left( \langle (t\otimes \id_{F^\adjoint})\xi |\zeta\rangle \right) = \eta^{-1}\left( \langle \xi | (t^*\otimes \id_{F^\adjoint})\zeta\rangle \right) = \langle \xi | \functor{F}_C(t^*)\zeta\rangle.
\]
Thus $\functor{F}_C(t)^*=\functor{F}_C(t^*)$ as maps of Hilbert $C^*$-modules over $A$, and this shows that $\functor{F}_C$ is a $*$-functor from $\ccomod^{\cb}(C)$ to $\cmod(A)$. 
\end{proof}

\begin{proof}[Proof of Theorem \ref{thm:Hermitian}]
Proposition \ref{prop:exp}, with $L=F$ and $R=F^\adjoint$, implies that for each $X\in \cmod(A)$ there is a completely isometric natural isomorphism of right operator $A$-modules $X\xrightarrow{\cong} \functor{F}^\adjoint_C\functor{F}_C(X)$; and that for each $(Z,\delta_Z)\in \ccomod(C)$ there is a completely isometric natural isomorphism of right operator $C$-comodules $(Z,\delta_Z)\xrightarrow{\cong} \functor{F}_C\functor{F}^\adjoint_C(Z,\delta_Z)$. These are isometric module isomorphisms between Hilbert $C^*$-modules, and so by a theorem of Lance (cf.~\cite[8.1.8]{BLM}) they are in fact unitary isomorphisms.
\end{proof}

\begin{corollary}\label{cor:coaction-cb-cc}
Suppose that $A$ acts faithfully by $B$-compact operators on a $C^*$-correspondence ${}_A F_B$, and that the inclusion $\eta:A\to\Compact_B(F)$ admits a weak expectation $\Compact_B(F)\to A^{\dual\dual}$. Then for the   coalgebra $C=F^\adjoint\otimes^{\h}_B F$ one has  $\ccomod^{\cb}(C)=\ccomod(C)$: every completely bounded Hilbert $C^*$-comodule over $C$ is in fact completely contractive.
\end{corollary}

\begin{proof}
Let $(Z,\delta_Z)\in \ccomod^{\cb}(C)$ be a completely bounded comodule. Proposition \ref{prop:exp}(b) shows that the image of the coaction $\delta_Z:Z\to Z\otimes^{\h}_B C$ is equal to the Hilbert $C^*$-module $(Z\boxtimes_C F^\adjoint)\otimes^{\h}_A F$. The $A$-valued inner product on $Z\boxtimes_C F^\adjoint$ was defined (in Lemma \ref{lem:boxtimes-ip}) by restriction from the $K$-valued inner product on $Z\otimes^{\h}_B F^\adjoint$, and it follows from this that the map
\[
(Z\boxtimes_C F^\adjoint)\otimes^{\h}_A F \xrightarrow{\xi\otimes f\mapsto \xi\otimes f}   Z\otimes^{\h}_B F^\adjoint \otimes^{\h}_K F
\]
is completely isometric. Now the map 
\[
F^\adjoint \otimes^{\h}_K F \to B,\qquad f_1^\adjoint\otimes  f_2\mapsto \langle f_1|f_2\rangle=\epsilon(f_1^\adjoint\otimes_A f_2)
\]
is also completely isometric: indeed, it is an isomorphism of $C^*$-algebras from the algebra of $K$-compact operators on $F^\adjoint$ to the closed ideal $\image(\epsilon)\subseteq B$. It follows from these two observations that the map $\epsilon_Z: Z\otimes^{\h}_B C\to Z$ is completely isometric on the image of $\delta_Z$. Since $\epsilon_Z\circ \delta_Z=\id_Z$, we conclude that $\delta_Z$ is also a complete isometry.
\end{proof}

\begin{remark}\label{rem:bimodule-cstar}
One can adapt Theorem \ref{thm:Hermitian} to categories of bimodules, similarly to what was explained in Remark \ref{rem:bimodules}. So consider, for $C^*$-algebras $A$, $B$ and $D$, the categories  $\cmod(D,A)$ and $\cmod(D,B)$ of $C^*$-correspondences from $D$ to $A$, and $C^*$-correspondences from $D$ to $B$. Let ${}_A F_B$ be a $C^*$-correspondence on which $A$ acts by $B$-compact operators, and consider the associated coalgebra $C=F^\adjoint\otimes^{\h}_A F$. The proof of Theorem \ref{thm:Hermitian} carries over verbatim to this setting, and gives for instance the following criterion for factoring operators on $C^*$-correspondences: with the same assumptions as Theorem \ref{thm:Hermitian}, if $Z\in \cmod(D,B)$ is a $C^*$-correspondence and $t:Z\to Z$ is an adjointable bimodule map, then there exists a correspondence $X\in \cmod(D,A)$, an adjointable bimodule map $s:X\to X$, and a unitary bimodule isomorphism $u:Z\to X\otimes^{\h}_A F$ with $t=u^*(s\otimes\id_F)u$, if and only if there exists a completely bounded $D$-$B$-linear coaction $\delta_Z:Z\to Z\otimes^{\h}_B C$ satisfying \eqref{eq:coaction-hermitian}, and having $\delta_Z\circ t = (t\otimes\id_C)\circ \delta_Z$. 
\end{remark}

\section{Examples}\label{sec:examples}

\subsection{Parabolic induction}\label{subsec:parabolic}

Let $G$ be a real reductive group, and let $P$ be a parabolic subgroup of $G$ with   unipotent radical $N$. The quotient $L\coloneqq P/N$ is another real reductive group. An example is $G=\GL_n(\R)$ with (for $n=a+b$)
\[
P=\begin{bmatrix} \GL_a(\R) & M_{a\times b}(\R) \\ 0  & \GL_b(\R)\end{bmatrix},\  N=\begin{bmatrix} 1_{a} & M_{a\times b}(\R) \\ 0 & 1_{b}\end{bmatrix},\ L\cong \GL_a(\R)\times \GL_b(\R).
\]

Clare \cite{Clare} showed how to construct a $C^*$-correspondence, which we shall denote here by ${}_{C^*_r(G)}C^*_r(G/N)_{C^*_r(L)}$, whose associated tensor-product functor
\[
\Ind_P:\starRep(C^*_r(L)) \to \starRep(C^*_r(G)),\qquad H\mapsto C^*_r(G/N)\otimes^{\h}_{C^*_r(L)} H
\]
is isomorphic to the well-known functor  of \emph{parabolic induction} of tempered unitary representations. This correspondence was further analysed in \cite{CCH-Compositio} and \cite{CCH-JIMJ}, where we showed that the adjoint operator bimodule 
\[
C^*_r(N\backslash G)\coloneqq {}_{C^*_r(L)}C^*_r(G/N)_{C^*_r(G)}^\adjoint
\] 
is completely boundedly isomorphic to a $C^*$-correspondence. This implies that the \emph{parabolic restriction} functor 
\[
\Res_P:\opmod(C^*_r(G))\to \opmod(C^*_r(L)),\qquad X\mapsto X\otimes^{\h}_{C^*_r(G)} C^*_r(G/N)
\]
sends Hilbert-space representations of $C^*_r(G)$ to Hilbert-space representations of $C^*_r(L)$, and hence furnishes a two-sided adjoint to $\Ind_P$ on tempered unitary representations. The arguments in \cite{CCH-Compositio} relied  on deep theorems from representation theory, and it would be of great interest to find an alternative, more geometric route to this adjunction theorem. 

%

The operator $C^*_r(L)$-bimodule 
\[
C_P\coloneqq C^*_r(N\backslash G)\otimes^{\h}_{C^*_r(G)} C^*_r(G/N)
\]
is of great interest from a representation-theoretic standpoint: indeed, it follows from \cite[1.5.14 \& 3.5.10]{BLM} that for every pair of tempered unitary representations $H_1,H_2$ of $L$, the space of intertwining operators between the parabolically induced representations $\Ind_P H_1$ and $\Ind_P H_2$ can be recovered as 
\[
\Bounded_G(\Ind_P H_1,\Ind_P H_2) \cong \left( H_2^\adjoint\otimes^{\h}_{C^*_r(L)} C_P \otimes^{\h}_{C^*_r(L)} H_1\right)^\dual.
\]
The results of this paper show that $C_P$ carries extra algebraic structure that is very relevant to studying the representations of $G$.

\begin{corollary}\label{cor:reductive}
The operator bimodule $C_P$ is an operator $C^*_r(L)$-coalgebra, and the parabolic restriction functor $\Res_P:X\to X\otimes^{\h}_{C^*_r(G)} C^*_r(G/N)$ induces a completely isometric equivalence of $\opsp_1$-categories
\[
\Res_P: \opmod(C^*_r(G)_P) \xrightarrow{\cong} \opcomod(C_P),
\]
and a unitary equivalence of $C^*$-categories
\[
\Res_P:\cmod(C^*_r(G)_P) \xrightarrow{\cong} \ccomod(C_P),
\]
where $C^*_r(G)_P$ denotes the image of $C^*_r(G)$ under the action homomorphism $C^*_r(G)\to  \Compact_{C^*_r(L)}(C^*_r(G/N))$.
\end{corollary}

\begin{proof}
It was observed by Clare that simple geometric considerations (the compactness of $G/P$) ensure that $C^*_r(G)$ acts on $C^*_r(G/N)$ through $C^*_r(L)$-compact operators. It follows that $( C^*_r(G/N), C^*_r(N\backslash G))$ is an adjoint pair of bimodules in the sense of Definition \ref{def:adjunction}, and so the bimodule $C_P$ is a $C^*_r(L)$-coalgebra as in Section \ref{subsec:adjoint-coalgebra}. The $C^*$-algebra $C^*_r(G)$ is nuclear---indeed, by a theorem of Harish-Chandra \cite[Theorem 7]{Harish-Chandra}, it is of type I---and hence its quotient $C^*_r(G)_P$ is also nuclear and thus has the weak expectation property. Since $C^*_r(G/N)$ is faithful as a $C^*_r(G)_P$-module  by definition, the asserted equivalences follow  from Lemma \ref{lem:WEP-faithful} and  Theorems \ref{thm:WEP} and \ref{thm:Hermitian}.
\end{proof}

\begin{remark}\label{rem:Frobenius}
The results of \cite{CCH-Compositio} and \cite{CCH-JIMJ} allow us to say more about the coalgebra $C_P$: the $C^*_r(G)$-valued inner product on $C^*_r(N\backslash G)$ constructed in \cite{CCH-Compositio} furnishes $C_P$ with a completely bounded multiplication $C_P\otimes^{\h}_{C^*_r(L)} C_P\to C_P$, and as an algebra $C_P$ is completely boundedly isomorphic to the $C^*$-algebra of $C^*_r(G)$-compact operators on $C^*_r(N\backslash G)$. The algebra and coalgebra structures on $C_P$ are compatible with one another: the coproduct is a map of $C_P$-bimodules, and the product is a map of $C_P$-bi-comodules; in other words $C_P$ is a (non-unital) \emph{Frobenius algebra} in the category of operator $B$-bimodules. Just as for Frobenius algebras in the purely algebraic setting (cf.~\cite{Abrams}), one has an equivalence $\opcomod^{\cb}(C_P)\cong \opmod^{\cb}(C_P)$ between the categories of (completely bounded) comodules and modules over $C_P$. Similar considerations apply to any locally adjoint pair of $C^*$-correspondences admitting a unit, in the terminology of \cite{CCH-JIMJ}. 
\end{remark}

\subsection{The Fourier coalgebra of a compact group}\label{subsec:Fourier}

To begin with we let $A\subseteq\Compact(H)$ be a $C^*$-algebra of compact operators, which is nondegenerate in the sense that $H=AH$. Regarding $H$ as a $C^*$-correspondence from $A$ to $\C$ puts us in the setting of Example \ref{ex:corresp}: the pair $({}_A H_{\C}, {}_\C H^\adjoint_A)$ is an adjoint pair of operator bimodules, with unit the inclusion map $A\to \Compact(H)\cong H\otimes^{\h}_{\C} H^\adjoint$ and counit the inner product $H^\adjoint\otimes^{\h}_A H\to \C$.  

The operator $\C$-coalgebra $C=H^\adjoint\otimes^{\h}_A H$ associated to this adjoint pair is isomorphic to one previously studied in \cite{Effros-Ruan-Hopf}, as we shall now explain. We note firstly that if $A\subseteq \Bounded(H)$ is any nondegenerate $C^*$-algebra of operators, then the pairing 
\[
(H^\adjoint\otimes^{\h}_A H) \times A' \to \C,\qquad \llangle \xi^\adjoint\otimes \zeta | t\rrangle \coloneqq  \langle \xi| t\zeta\rangle \qquad (\xi,\zeta\in H,\ t\in A')
\]
induces a completely isometric isomorphism of operator spaces
\begin{equation}\label{eq:predual}
H^\adjoint\otimes^{\h}_A H \xrightarrow{\cong} A'_{\dual}
\end{equation}
between the Haagerup tensor product $H^\adjoint\otimes^{\h}_A H$ and the predual   of the commutant $A'\subseteq \Bounded(H)$. This is a special case of \cite[Corollary 3.5.10]{BLM}.  

In \cite{Effros-Ruan-Hopf}, Effros and Ruan showed that the predual $M_\dual$ of any von Neumann algebra $M$ carries a coalgebra structure dual to the algebra structure of $M$: the inclusion-of-the-unit map $u:\C\to M$ dualises to give a counit $u_\dual:M_\dual\to \C$, while the multiplication map $\mu:M\otimes^{\sigmah}M\to M$ dualises to give a coproduct $\mu_\dual:M_\dual\to M_\dual\otimes^{\eh} M_\dual$.  Here $\otimes^{\sigmah}$ and $\otimes^{\eh}$ denote the \emph{normal} and the \emph{extended} Haagerup tensor products, respectively. We refer the reader to \cite{Effros-Ruan-Hopf} for details.

\begin{lemma}\label{lem:Effros-Ruan-iso}
If $A\subseteq \Compact(H)$ is a nondegenerate $C^*$-algebra of compact operators, then Effros and Ruan's coproduct $\mu_\dual:A'_{\dual}\to A'_{\dual}\otimes^{\eh} A'_{\dual}$ has image contained in $A'_{\dual}\otimes^{\h} A'_{\dual}$, and the map \eqref{eq:predual} is a completely isometric isomorphism of operator $\C$-coalgebras.
\end{lemma}

\begin{proof}
Let us denote the isomorphism \eqref{eq:predual} by $\phi:C\to A'_\dual$. We are claiming that the diagrams 
\[
\xymatrix@C=40pt{
C\ar[r]^-{\phi} \ar[d]_-{\delta} & A'_\dual \ar[d]^-{\mu_\dual} \\
C\otimes^{\h} C \ar[r]^-{i\circ(\phi\otimes\phi)} & A'_\dual\otimes^{\eh} A'_\dual
}
\qquad \text{and}\qquad 
\xymatrix@C=40pt{
C\ar[r]^-{\phi} \ar[d]_-{\epsilon} & A'_\dual \ar[dl]^-{u_\dual} \\
\C & 
}
\]
commute, where $i:A'_\dual\otimes^{\h} A'_\dual \to A'_\dual\otimes^{\eh} A'_\dual$ is the canonical inclusion. 

By the definition of the coproduct $\mu_\dual$, and by the fact that the image of the algebraic tensor product $A'\otimes A'$ is weak$^*$-dense in $A'\otimes^{\sigmah}A'$ (see \cite[Lemma 5.8]{Effros-Ruan-Hopf}), to show that the first diagram commutes it will suffice to prove that 
\begin{equation*}\label{eq:duality-pf-1}
\llangle \phi(\xi^\adjoint\otimes a\zeta) | t_1 t_2 \rrangle = \llangle i\circ(\phi\otimes\phi)\circ\delta(\xi^\adjoint\otimes a\zeta)| t_1\otimes t_2\rrangle
\end{equation*}
for all $\xi,\zeta\in H$, all $a\in A$ and all $t_1,t_2\in A'$ (where $\llangle \cdot,\cdot\rrangle$ denote the duality pairings). Approximating $a$ in norm by finite-rank operators 
\[
a = \lim_{n\to \infty} \sum_{m=1}^{r_n} \alpha_{n,m}\otimes \beta_{n,m}^\adjoint  \qquad (\alpha_{n,m},\beta_{n,m}\in H) 
\]
we have 
\[
 \delta(\xi^\adjoint\otimes a\zeta) = \lim_n \sum_m  (\xi^\adjoint\otimes \alpha_{n,m})\otimes  (\beta_{n,m}^\adjoint\otimes \zeta),
\]
and so
\[
\begin{aligned}
& \llangle i\circ(\phi\otimes\phi)\circ\delta(\xi^\adjoint\otimes a\zeta) | t_1\otimes t_2\rrangle \\ & = 
\lim_n \sum_m \llangle \phi(\xi^\adjoint\otimes \alpha_{n,m}) | t_1\rrangle \llangle \phi(\beta_{n,m}^\adjoint\otimes \zeta), t_2\rrangle \\
& =\lim_n\sum_m \langle \xi | t_1\alpha_{n,m}\rangle \langle \beta_{n,m} | t_2\zeta\rangle \\
& = \langle \xi | t_1 a t_2\zeta\rangle = \langle \xi | t_1 t_2 a\zeta\rangle   = \llangle \phi(\xi^\adjoint\otimes a\zeta) | t_1 t_2\rrangle.
\end{aligned}
\]
Thus the first diagram commutes.

The commutativity of the second diagram is easily seen by noting that  
\[
\epsilon(\xi^\adjoint\otimes\zeta) = \langle \xi | \zeta\rangle = \llangle \phi(\xi^\adjoint\otimes\zeta) | 1\rrangle = u_\dual(\phi(\xi^\adjoint\otimes \zeta)).\qedhere
\]
\end{proof}

Since $A$ has the weak expectation property (indeed, $A$ is nuclear), Theorem \ref{thm:WEP} gives the following identification of $A$-modules and $A'_{\dual}$-comodules:

\begin{corollary}\label{cor:Effros-Ruan}
If $A\subseteq\Compact(H)$ is a nondegenerate $C^*$-algebra of compact operators then the comparison functor
\[
\opmod(A)\to \opcomod(A'_\dual),\qquad X\mapsto (X\otimes^{\h}_A H, \eta_X\otimes \id_H)
\]
is a completely isometric equivalence.\hfill\qed
\end{corollary}

Now let $G$ be a compact group, and consider the right-regular representation $\rho$ of $G$ on $L^2(G)$. As is well known, the commutant $\rho(G)'$ is isomorphic to the group von Neumann algebra of $G$, whose predual $\rho(G)'_\dual$ is the \emph{Fourier algebra} $A(G)$ of continuous functions $G\to \C$ which are matrix coefficients of the left-regular representation $\lambda$: i.e., those functions of the form
\[
c_{\xi^\adjoint\otimes\zeta}: g\mapsto \langle \xi | \lambda(g)\zeta \rangle_{L^2(G)}
\]
for some $\xi,\zeta\in L^2(G)$. The duality pairing between an operator $t\in \rho(G)'$ and a  function $c_{\xi^\adjoint\otimes \zeta}\in A(G)$ is given explicitly by 
\[ 
\llangle c_{\xi^\adjoint\otimes\zeta} | t\rrangle =  \langle \xi | t \zeta \rangle_{L^2(G)}.
\]
See \cite{Eymard}. 

Putting $A=C^*(G)$ and $H=L^2(G)$ (on which $C^*(G)$ acts via $\rho$) in the discussion above, we find that the map
\begin{equation}\label{eq:Fourier-iso}
L^2(G)^\adjoint\otimes^{\h}_{C^*(G)} L^2(G) \to A(G),\qquad \xi^\adjoint\otimes\zeta\mapsto c_{\xi^\adjoint\otimes\zeta}
\end{equation}
is a completely isometric isomorphism of operator $\C$-coalgebras, where the coproduct on $A(G)$ is dual to the multiplication in $G$, and the counit is evaluation at the identity. Corollary \ref{cor:Effros-Ruan} then gives the following identification between operator $C^*(G)$-modules and operator $A(G)$-comodules:

\begin{corollary}\label{cor:G-compact}
For each compact group $G$ the functor 
\[
\opmod(C^*(G))\to \opcomod(A(G)),\qquad X\mapsto X \otimes^{\h}_{C^*(G)} L^2(G)
\]
is a completely isometric equivalence.\hfill\qed
\end{corollary}

\begin{remark}
It would certainly be interesting to extend Corollary \ref{cor:G-compact} beyond the setting of compact groups. The identification \eqref{eq:Fourier-iso} is valid for every locally compact group $G$; but when $G$ is not compact the coproduct on $A(G)$ takes values in the \emph{extended} Haagerup tensor product, and so this case lies beyond the scope of Theorem \ref{thm:WEP}. It may be possible to remedy this by extending our framework to include coproducts and coactions taking values in spaces of multipliers, as is frequently done in the setting of $C^*$-algebraic quantum groups, and as is suggested in the present context by Daws in \cite[Section 9.3]{Daws}.
\end{remark}

\subsection{Subalgebras and flat connections}\label{subsec:connections}

In this  section we study nondegenerate inclusions of $C^*$-algebras $A\into B$ (nondegeneracy meaning that $B=AB=BA$). We consider the $B$-coalgebra $C=B\otimes^{\h}_A B$ associated to the adjoint pair of operator bimodules $({}_A B_B, {}_B B_A)$, with unit $\eta:A\to B\cong B\otimes^{\h}_B B$ the inclusion, and counit $B\otimes^{\h}_A B\to B$ the product. Considering $B$ as a Hilbert $C^*$-module over itself, with inner product $\langle b_1|b_2\rangle= b_1^* b_2$, puts us in the context of Section \ref{sec:descent-Cmod}. In this section we shall  adapt some algebraic observations of Nuss \cite{Nuss} and Brzezinski \cite{Brz-grouplike} to this $C^*$-algebraic setting to give a description of the $C^*$-comodule category  $\ccomod(C)$   in terms of flat connections on Hilbert $C^*$-modules over $B$. (A similar description applies to the category $\opcomod(C)$ of operator comodules. These categories may likewise be described in terms of \emph{descent data} and \emph{involutive twists}, as is done in the algebraic setting in \cite{Cipolla} and \cite{Nuss}.)

Let $\Omega=\Omega(B,A)\coloneqq \ker \epsilon\subseteq C$ denote the kernel of the multiplication map. This is the analogue in the operator-algebraic setting of the module of relative one-forms of \cite{Cuntz-Quillen-extensions}. The case where $A=\C$ has previously been considered by Mesland in \cite{Mesland}. Note that $\Omega$ is stable under the involution $(b_1\otimes b_2)^*\coloneqq b_2^*\otimes b_1^*$ on $C$.

We do not assume that the algebra $B$ has a unit; but of course $B$ can be embedded completely isometrically in a unital $C^*$-algebra $B^+$, and Theorem \ref{thm:exact} ensures that this embedding induces completely isometric embeddings of Haagerup tensor products $B\otimes^{\h}_A B\into B^+\otimes^{\h}_A B^+$ (and likewise for higher tensor powers). Note that the nondegeneracy of $B$ over $A$ ensures that elements of the form $b\otimes 1,1\otimes b\in B^+\otimes^{\h}_A B^+$ belong to the submodule $B\otimes^{\h}_A B$. 

We consider the completely bounded $A$-bimodule maps
\[
\begin{aligned}
&d:B\to \Omega,\quad  b\mapsto 1\otimes b - b\otimes 1 \quad \textrm{and} \\
& d^1:\Omega\to \Omega\otimes^{\h}_B \Omega,\quad  \omega\mapsto  q\circ(d\otimes d)\restrict_{\Omega}(\omega),
\end{aligned}
\]
where $q:\Omega\otimes^{\h}_A \Omega\to \Omega\otimes^{\h}_B \Omega$ is the quotient mapping.
These maps satisfy the Leibniz rules 
\[
\begin{aligned}
& d(bb')=d(b)b'+bd(b'),\\
& d^1(b\omega)=d(b)\otimes\omega + bd^1(\omega),\\
& d^1(\omega b) = d^1(\omega)b-\omega\otimes d(b)
\end{aligned}
\]
for all $b,b'\in B$ and all $\omega\in \Omega$, and one has $d^1\circ d=0$. Moreover, $d(b^*)=-d(b)^*$ and $d^1(\omega^*)=d^1(\omega)^*$ for all $b\in B$ and $\omega\in  \Omega$.
(This structure may be extended to a differential graded $*$-algebra, but we won't need to consider that here.) 

\begin{definition}\label{def:connection} (cf.~\cite[Section 8]{Cuntz-Quillen-extensions}, \cite[Section 5]{Mesland})
An \emph{$\Omega$-connection} on a right operator $B$-module $Z$ is a completely bounded $A$-linear map $\nabla:Z\to Z\otimes^{\h}_B \Omega$  satisfying
\[
\nabla(zb)=\nabla(z)b + z\otimes d(b)
\]
for all $b\in B$ and $z\in Z$. The \emph{curvature}  of a connection $\nabla$ is the composite map 
\[
Z\xrightarrow{\nabla} Z \otimes^{\h}_B \Omega \xrightarrow{\id_Z\otimes d^1 + \nabla\otimes \id_\Omega} Z\otimes^{\h}_B \Omega \otimes^{\h}_B\Omega,
\] 
and a \emph{flat} connection is one whose curvature is zero. If $Z$ is a right Hilbert $C^*$-module over $B$, then a connection $\nabla:Z\to Z\otimes^{\h}_B \Omega$ is called \emph{Hermitian} if it satisfies
\begin{equation}\label{eq:nabla-Hermitian}
\langle z_1 | \nabla(z_2)\rangle_{\Omega} - \langle z_2 | \nabla(z_1)\rangle_{\Omega}^* = d(\langle z_1|z_2\rangle)
\end{equation}
for all $z_1,z_2\in Z$, where we are using the pairing
\[
 \langle \cdot|\cdot\rangle_\Omega: Z\times (Z\otimes^{\h}_B \Omega)\to \Omega,\qquad \langle z_1| z_2\otimes\omega\rangle_{\Omega}\coloneqq \langle z_1|z_2\rangle \omega.
\]

Let $\conn (B,A)$ denote the category whose objects $(Z,\nabla_Z)$ are right Hilbert $C^*$-modules over $B$ equipped with flat, Hermitian $\Omega(B,A)$-connections, with morphisms $t:(W,\nabla_W)\to (Z,\nabla_Z)$   the adjointable maps of Hilbert $C^*$-modules satisfying $\nabla_Z\circ t = (t\otimes \id_\Omega)\circ \nabla_W$.  
\end{definition}

\begin{example}
Every countably generated Hilbert $C^*$-module $Z$ over $B$ may be equipped with a  Hermitian $\Omega(B,A)$-connection: the construction of  {Grassmann connections} in \cite[Proposition 8.1]{Cuntz-Quillen-extensions} and \cite[Corollary 5.15]{Mesland} carries over verbatim to this setting. 
\end{example} 
 
The category $\conn(B,A)$ is a $C^*$-category: this can be shown by a direct argument as in Lemma 5.3, but it also follows easily from the following equivalence between flat connections and comodules over $C=B\otimes^{\h}_A B$, which is proved just as in the algebraic case (cf.~\cite[Theorem 4.4]{Brz-grouplike}).

\begin{lemma}\label{lem:comod-conn} 
Let $(Z,\delta_Z)$ be a completely bounded Hilbert $C^*$-comodule over $C$. The map $\nabla_{\delta_Z}$ defined as the composition 
\[
 Z \xrightarrow{\delta_Z} Z\otimes^{\h}_B C \xrightarrow{\id_Z\otimes \theta} Z\otimes^{\h}_B\Omega,
\]
where $\theta(c)\coloneqq c-\epsilon(c)\otimes 1$, is a flat Hermitian connection. Conversely, if $\nabla_Z$ is a flat Hermitian $\Omega(B,A)$-connection on a right Hilbert $C^*$-module $Z$ over $B$, then the map $\delta_{\nabla_Z}$ defined by 
\[
 \delta_{\nabla_Z}: Z\to Z\otimes^{\h}_B C,\quad z\mapsto \nabla_Z(z) + z\otimes 1\otimes 1
\]
gives $Z$ the structure of a completely bounded Hilbert $C^*$-comodule over $C$. The functors
\[
\begin{aligned}
&\ccomod^{\cb}(C) \xrightarrow{ \functor{E}} \conn(B,A),\quad \functor{E}(Z,\delta_Z)= (Z,\nabla_{\delta_Z}),\quad \functor{E}=\id\textrm{ on morphisms}\\
&\conn(B,A)\xrightarrow{\functor{D}} \ccomod^{\cb}(C),\quad \functor{D}(Z,\nabla_Z) = (Z,\delta_{\nabla_Z}),\quad \functor{D}=\id\textrm{ on morphisms}
\end{aligned}
\]
are mutually inverse isomorphisms of categories.\hfill\qed
\end{lemma}

\begin{corollary}\label{cor:cmod-conn}
Let $A$ be a  $C^*$-algebra, embedded as a nondegenerate subalgebra of a $C^*$-algebra $B$, and suppose that there is a weak expectation $B\to A^{\dual\dual}$. Then the functors 
\[
\cmod(A)\to \conn(B,A),\qquad X\mapsto (X\otimes^{\h}_A B, \nabla(x\otimes b)\coloneqq x\otimes 1\otimes d(b))
\]
and
\[
\conn(B,A) \to \cmod(A),\qquad (Z,\nabla_Z)\mapsto \ker(\nabla_Z)
\]
are mutually inverse unitary equivalences of $C^*$-categories.
\end{corollary}

\begin{proof}
First note that Corollary \ref{cor:coaction-cb-cc} gives an equality $\ccomod^{\cb}(C)=\ccomod(C)$. Now the functors in question are equal, respectively, to the compositions
\[
\cmod(A) \xrightarrow{\functor{F}_C}  \ccomod(C) \xrightarrow{\functor{E}} \conn(B,A)\ \quad \textrm{and}\quad \conn(B,A) \xrightarrow{\functor{D}} \ccomod(C)   \xrightarrow{\functor{F}^*_C} \cmod(A) 
\]
and so the corollary follows from Theorem \ref{thm:Hermitian} and Lemma \ref{lem:comod-conn}.
\end{proof}

\begin{example}
Setting $A=\C$ we find that for a unital $C^*$-algebra $B$, the Hilbert $C^*$-modules over $B$  admitting a flat, Hermitian $\Omega(B,\C)$-connection are precisely the `free modules' $B^{\oplus n}$ (where $n$ may be infinite). 
\end{example}

\begin{remarks}\begin{enumerate}[\rm(a)]
\item Once again (cf.~Remark \ref{rem:bimodule-cstar}) there is a bimodule version of Corollary \ref{cor:cmod-conn}, relating $C^*$-correspondences ${}_D E_A$ to  flat, $D$-linear, Hermitian connections on $C^*$-correspondences ${}_D F_B$, for each $C^*$-algebra $D$.   
\item It follows from Corollary \ref{cor:cmod-conn} that the $K$-theory of $A$ can be described in terms of the category $\conn(B,A)$, in such a way that the map $K_*(A)\to K_*(B)$ induced by the inclusion $A\to B$ corresponds to the one induced by the forgetful functor $\conn(B,A)\to\cmod(B)$. 
\end{enumerate}
\end{remarks}

\subsection{The maximal ${C^*}$-dilation}\label{subsec:Cstarmax}

For our final example  we consider the descent problem for the inclusion $A\into C^*  A$ of a not-necessarily-self-adjoint operator algebra $A$ into its \emph{maximal $C^*$-algebra}. This embedding is characterised by the   property that if $D$ is a $C^*$-algebra generated by a completely isometrically embedded copy of $A$, then the identity map on $A$ extends uniquely to a $*$-homomorphism $C^*  A\to D$. If $A$ is a $C^*$-algebra then $C^* A=A$ and the inclusion is the identity. See \cite{Blecher-cstarmax}. 

The passage from (modules over) $A$ to (modules over) $C^*  A$ was used in \cite{Blecher-Morita} and \cite{Blecher-Solel} as a way to relate the representation theory of non-self-adjoint operator algebras to the better understood $C^*$-algebra theory. 
Given an operator algebra $A$, one considers the \emph{maximal $C^*$-dilation} functor
\[
\functor{L}:\opmod(A) \to \opmod(C^*  A),\qquad X\mapsto X\otimes^{\h}_A C^*  A,
\]
which is left-adjoint to the forgetful functor $\functor{R}:\opmod(C^*  A)\to \opmod(A)$. This is the pair of functors corresponding to the adjoint pair of operator bimodules $({}_A C^*  A_{C^*_{\csmax }A}, {}_{C^*  A} C^*  A_ A)$ (cf.~Example \ref{ex:subalg}). Letting $C$ denote the associated operator $C^*  A$-coalgebra $C^* A\otimes^{\h}_A C^* A$, we have a comparison functor
\[
\functor{L}_C:\opmod(A)\to \opcomod(C), \qquad X\mapsto (X\otimes^{\h}_A C^*  A, \eta_X\otimes \id_{C^*  A}).
\]
Theorems of Beck and of Blecher imply that $\functor{L}_C$ is completely isometrically fully faithful: 

\begin{proposition}\label{prop:cstarmax-fully-faithful}
For each pair of operator $A$-modules $X,Y\in \opmod(A)$ the natural map
\[
\functor{L}_C: \CB_A(X,Y) \to \CB_C(\functor{L}_C(X),\functor{L}_C(Y))
\]
is a completely isometric isomorphism.
\end{proposition}

\begin{proof}
The functor $\functor{L}_C$ is given on morphisms by a Haagerup tensor product, so it is completely contractive. Blecher showed in \cite[Corollary 3.11]{Blecher-cstarmax} that for each $X\in \opmod(A)$ the natural map 
\[
\eta_X:X\to X\otimes^{\h}_A C^*  A,\qquad xa\mapsto x\otimes a
\]
is a completely isometric embedding, and since for each $t\in \CB_A(X,Y)$ the diagram
\[
\xymatrix@C=80pt{
X \ar[r]^-{t} \ar[d]_-{\eta_X} & Y \ar[d]^-{\eta_Y} \\
X\otimes^{\h}_A C^*  A \ar[r]^-{\functor{L}_C(t) = t\otimes \id_{C^*  A}} & Y\otimes^{\h}_A C^*  A
}
\]
commutes, this implies that the map $t\mapsto \functor{L}_C(t)$ is in fact a complete isometry.

The proof that the map $\functor{L}_C$ is surjective is a special case of an argument due to Beck (cf.~\cite[Theorem 3.13]{ttt}). Using \cite[Corollary 3.11]{Blecher-cstarmax} once again, one shows as in the proof of Proposition \ref{prop:exp}(a)  that 
\[
\eta_X(X) =  \left\{\left. \xi\in X\otimes^{\h}_A C^*  A \ \right| \ (\id_X\otimes \eta_{C^*  A})(\xi) = (\eta_X\otimes \id_{C^*  A})(\xi)\right\}
\]
for every $X\in \opmod(A)$. Using this, and the naturality of $\eta$, one shows that each comodule map $s\in \CB_C(\functor{L}_C(X),\functor{L}_C(Y))$ restricts to a map $\eta_X(X)\to \eta_Y(Y)$. A simple computation then shows that $s=t\otimes \id_{C^* A}$ for the map $t\coloneqq \eta_Y^{-1}\circ s \circ \eta_X\in \CB_A(X,Y)$.
\end{proof}

Proposition \ref{prop:cstarmax-fully-faithful} implies that the category $\opmod(A)$ can be identified, completely isometrically, with a full subcategory of $\opcomod(C)$. In the interests of clarifying the relationship between the representation theories of $A$ and of $C^*  A$, it would be useful to characterise this subcategory. In this section we shall prove a first result in this direction: if $A$ is not a $C^*$-algebra then the subcategory in question is a \emph{proper} subcategory of $\opcomod(C)$.

\begin{theorem}\label{thm:Cstar-max}
Let $A$ be a locally unital operator algebra. Then $A$ is a $C^*$-algebra if and only if the comparison functor $\functor{L}_C:\opmod(A)\to \opcomod(C)$ for the operator $C^*  A$-coalgebra $C=C^* A\otimes^{\h}_A C^* A$ is  a completely isometric equivalence.
\end{theorem}

We shall prove Theorem \ref{thm:Cstar-max} as a corollary of the following general property of operator coalgebras over $C^*$-algebras. Recall from Section \ref{subsec:opsp1} that for an operator $B$-coalgebra $C$, we denote by $\opcomod(C)_1$ and $\opmod(B)_1$ the subcategories of completely contractive maps in $\opcomod(C)$ and $\opmod(B)$ (respectively). Let us say   that a pair of morphisms $(i,q)$ is a \emph{kernel-cokernel pair} if $i$ is a kernel of $q$ and $q$ is a cokernel of $i$.

\begin{lemma} \label{lem:F-exact}
Let $B$ be a $C^*$-algebra and let $C$ be an operator $B$-coalgebra. Then the forgetful functors $\functor{F}:\opcomod(C)\to \opmod(B)$ and $\functor{F}_1:\opcomod(C)_1\to \opmod(B)_1$ preserve kernel-cokernel pairs. 
\end{lemma}

\begin{proof}
We shall present the proof for  $\functor{F}_1$, from which the proof for $\functor{F}$ differs only in notation. Let $i:(I,\delta_I)\to (Z,\delta_Z)$ and $q:(Z,\delta_Z)\to (Q,\delta_Q)$ be a kernel-cokernel pair in $\opcomod(C)_1$. The functor $\functor{F}_1$ has a right adjoint---namely, the `free comodule' functor $Y\mapsto (Y\otimes^{\h}_B C, \id_Y\otimes \delta)$---and so $\functor{F}_1$ preserves cokernels \cite[V.5]{MacLane}. Therefore $\functor{F}_1(q)$ is a cokernel of $\functor{F}_1(i)$ in $\opmod(B)_1$. 

To prove that $\functor{F}_1(i)$ is a kernel of $\functor{F}_1(q)$ in $\opmod(B)_1$, we will show that the map $q$ has a kernel $j:(J,\delta_J)\to (Z,\delta_Z)$ in $\opcomod(C)_1$ such that $\functor{F}_1(j)$ is a kernel of $\functor{F}_1(q)$ in $\opmod(B)_1$. The uniqueness of kernels in $\opcomod(C)_1$ will then imply that $i$ and $j$ are conjugate via an $\opcomod(C)_1$ isomorphism $h:(I,\delta_I)\to (J,\delta_J)$, hence that $\functor{F}_1(i)$ and $\functor{F}_1(j)$ are conjugate via the $\opmod(B)_1$ isomorphism $\functor{F}_1(h)$, and hence that $\functor{F}_1(i)$ is a kernel of $\functor{F}_1(q)$ in $\opmod(B)_1$ as required.

Let $J$ denote the operator $B$-submodule $\ker(q)\subseteq Z$, and let $j:J\to Z$ denote the inclusion map. Note that $j$ is a kernel of $q$ in $\opmod(B)_1$, by Lemma \ref{lem:ker-coker}. 

Because $B$ is a $C^*$-algebra, Theorem \ref{thm:exact} implies that the map $j\otimes \id_C:J\otimes^{\h}_B C\to Z\otimes^{\h}_B C$ is a completely isometric embedding, which we shall use to regard $J\otimes^{\h}_B C$ as a closed submodule of $Z\otimes^{\h}_B C$. Since $q$ is a map of $C$-comodules, the diagram
\[
\xymatrix@C=40pt{
J \ar[r]^-{\delta_Z|_J} \ar[d]_-{q|_J=0} & Z\otimes^{\h}_B C \ar[d]^-{q\otimes \id_C} \\
Q \ar[r]^-{\delta_Q} & Q\otimes^{\h}_B C
}
\]
commutes, showing that $\delta_Z (J)\subseteq \ker(q\otimes \id_C)$. Since the map $q:Z\to Q$ is a cokernel in $\opmod(B)_1$, Lemma  \ref{lem:ker-coker} and Theorem \ref{thm:exact} give $\ker(q\otimes \id_C)=\image(j\otimes \id_C)=J\otimes^{\h}_B C$. We may therefore regard $\delta_Z\restrict_J$ as a completely contractive map $J\to J\otimes^{\h}_B C$, furnishing $J$ with the structure of a $C$-comodule such that the inclusion $j:J\to Z$ is a map of comodules. 

It remains to check that $j:(J,\delta_Z\restrict_J)\to (Z,\delta_Z)$ is a kernel of $q$ in $\opcomod(C)_1$, which is straightforward: if $f:(W,\delta_W)\to (Z, \delta_Z)$ is a morphism in $\opcomod(C)_1$ with $q\circ f=0$, then the image of $f$ must lie in $J=\ker(q)$, and the   restricted map $f:V\to J$ is still a morphism in $\opcomod(C)_1$.
\end{proof}
 
\begin{corollary}\label{cor:L-exact}
Let $({}_A L_B, {}_B R_A)$ be an adjoint pair of operator bimodules, where $B$ is a $C^*$-algebra, and let $C=L\otimes^{\h}_A R$ be the associated $B$-coalgebra. If the comparison functor $\functor{L}_C:\opmod(A)\to \opcomod(C)$ is a completely isometric equivalence, then the functor $\functor{L}:\opmod(A)\to \opmod(B)$ preserves complete isometries. 
\end{corollary}  
  
\begin{proof}
Let $i:X\to Y$ be a completely isometric map of operator $A$-modules. Letting $q:Y\to Y/i(X)$ be the quotient map, Lemma \ref{lem:ker-coker} ensures that $(i,q)$ is a kernel-cokernel pair in $\opmod(A)_1$. If the functor $\functor{L}_C$ is a completely isometric equivalence, then it  restricts to an equivalence $\functor{L}_{C,1}:\opmod(A)_1\to \opcomod(C)_1$, and so $(\functor{L}_{C,1}(i),\functor{L}_{C,1}(q))$ is a kernel-cokernel pair in $\opcomod(C)_1$. Now Lemma \ref{lem:F-exact} implies that the map $\functor{F}_1(\functor{L}_{C,1}(i)) = \functor{L}(i)$ is a kernel in $\opmod(B)_1$ and hence, by Lemma \ref{lem:ker-coker}, a complete isometry.
\end{proof}

\begin{remark}\label{rem:L-exact-cb}
Replacing $\functor{F}_1$ by $\functor{F}$ in the proof of Corollary \ref{cor:L-exact} gives the following `completely bounded' variant: if $\functor{L}_C$ is an equivalence then $\functor{L}$ preserves complete embeddings.
\end{remark}
 
Theorem \ref{thm:Cstar-max} follows easily from Corollary \ref{cor:L-exact} and from a result of Blecher:

\begin{proof}[Proof of Theorem \ref{thm:Cstar-max}]
  If $\functor{L}_C$ is  a completely isometric equivalence then Corollary \ref{cor:L-exact} implies that the dilation functor $\functor{L}:\opmod(A)\to \opmod(C^* A)$ preserves complete isometries. Blecher has shown that the latter property holds (if and) only if $A$ is a $C^*$-algebra: see \cite[Theorem 4.3]{Blecher-cstarmax}. Conversely, if $A$ is a $C^*$-algebra then $C^* A=A$, and Theorem \ref{thm:WEP} ensures that the comparison functor $\functor{L}_C:\opmod(A)\to \opcomod(A\otimes^{\h}_A A)$ is  an equivalence.
\end{proof}

\begin{remark}\label{rem:Cstar-max-cb}
We can establish the following partial analogue of Theorem \ref{thm:Cstar-max} in the completely bounded setting: if the comparison functor $\functor{L}_C: \opmod(A)\to \opcomod(C)$ is a completely bounded equivalence,   then the algebra $A$ has the  \emph{module complementation property}: if $A\to \Bounded(H)$ is a completely contractive homomorphism and $H'\subseteq H$ is a closed $A$-invariant subspace, then $H'$ is topologically complemented in $H$. (See \cite[Section 7.2]{BLM} for a discussion of this property.) The proof is similar to the proof of Theorem \ref{thm:Cstar-max}: if $\functor{L}_C$ is an equivalence then the dilation functor preserves complete embeddings (see Remark \ref{rem:L-exact-cb}), and then replacing complete isometries by complete embeddings in Blecher's proof of \cite[Theorem 4.2]{Blecher-cstarmax} shows that $A$ has the module complementation property. It is known that for certain classes of operator algebras, the module complementation property is equivalent to being a $C^*$-algebra: see \cite[7.2.7, 7.2.10]{BLM}  for example. On the other hand, \cite{Choi-Farah-Ozawa} provides an example of a non-self-adjoint operator algebra with the module complementation property.  
\end{remark}

\begin{remark}\label{rem:monadic}
The pair $(_A C^*  A_{C^*  A}, {}_{C^*  A} C^*  A_A)$, where $A$ is not a $C^*$-algebra, is an example of an adjoint pair of bimodules for which the left adjoint functor $\functor{L}$ is not \emph{comonadic} (that is to say, $\functor{L}_C$ is not a completely isometric equivalence), while the right adjoint functor $\functor{R}$ is \emph{monadic}: indeed, the monad associated to this adjunction is the functor on $\opmod(A)$ of tensor product with the $A$-bimodule $K=C^*  A\otimes^{\h}_{C^*  A} C^*  A\cong C^*  A$; the associated Eilenberg-Moore category is just the category $\opmod(C^*  A)$; and the comparison functor
\[
\functor{R}_K:\opmod(C^* A)\xrightarrow{Y\mapsto Y\otimes^{\h}_{C^*  A}C^*  A} \opmod(C^*  A)
\]
is obviously a completely isometric equivalence.
\end{remark}

\bibliographystyle{alpha}
\bibliography{descent}

 \end{document}